\documentclass[12pt]{amsart}

\usepackage{geometry} 
\usepackage{amsmath, amsthm, amssymb, mathrsfs, mathtools, manfnt, xcolor}
\usepackage[colorlinks=true, linkcolor={red}, urlcolor={blue}, citecolor={blue}]{hyperref}
\usepackage[all, cmtip]{xy}

\newcommand{\ie}{i.e.\ }
\newcommand{\eg}{e.g.\ }
\newcommand{\inv}{^{-1}}

\newcommand{\F}{\mathbb F}
\newcommand{\D}{\mathbb D}
\newcommand{\E}{\mathbb E}

\newcommand{\ul}{\text{\smash{\raisebox{-1ex}{\scalebox{1.5}{$\ulcorner$}}}}}

\DeclareMathOperator{\dia}{dia}

\DeclareMathOperator{\Ho}{Ho}

\newcommand{\op}{^\text{op}}
\newcommand{\pr}{\text{pr}}

\newcommand{\swtrans}{\mathbin{\rotatebox[origin=c]{225}{$\Rightarrow$}}}
\newcommand{\netrans}{\mathbin{\rotatebox[origin=c]{45}{$\Rightarrow$}}}

\newcommand{\sU}{\mathscr U}

\newcommand{\cE}{\mathcal E}
\newcommand{\cW}{\mathcal W}
\newcommand{\cA}{\mathcal A}

\newcommand{\cC}{\mathcal C}

\newcommand{\cM}{\mathcal M}

\newcommand{\Cat}{\mathbf{Cat}}
\newcommand{\CAT}{\mathbf{CAT}}
\newcommand{\PDer}{\mathbf{PDer}}
\newcommand{\Der}{\mathbf{Der}}

\DeclareMathOperator{\Hom}{Hom}
\DeclareMathOperator{\Fun}{Fun}
\DeclareMathOperator{\id}{id}

\DeclareMathOperator{\colim}{colim}
\newcommand{\Dia}{\mathbf{Dia}}
\newcommand{\Dirf}{\mathbf{Dir_f}}

\renewcommand{\phi}{\varphi}

\numberwithin{equation}{section}
\numberwithin{figure}{section}

\theoremstyle{definition}
\newtheorem{theorem}[equation]{Theorem}
\newtheorem{prop}[equation]{Proposition}
\newtheorem{defn}[equation]{Definition}
\newtheorem{lemma}[equation]{Lemma}
\newtheorem{remark}[equation]{Remark}
\newtheorem{cor}[equation]{Corollary}

\newtheorem{notn}[equation]{Notation}

\newtheorem{ex}[equation]{Example}

\begin{document}

\title{The theory of half derivators}
\author{Ian Coley}
\begin{abstract}
We review the theory of derivators from the ground up, defining new classes of derivators which were originally motivated by derivator K-theory. We prove that many old arguments that relied on homotopical bicompleteness hold also for one-sided \emph{half derivators} on arbitrary diagram categories. We end by defining the maximal domain for a K-theory of derivators generalising Waldhausen K-theory.
\end{abstract}

\maketitle

\tableofcontents

\section{Introduction}\label{sec:intro}

The theory of derivators was developed initially (under different names) by Heller in \cite{Hel88},  Grothendieck in \cite{Gro90}, and (in the triangulated setting) Franke in \cite{Fra96}. In brief, a derivator represents an abstract bicomplete homotopy theory; we attach the adjective \emph{triangulated} to a derivator when it represents a stable (bicomplete) homotopy theory. Unlike the world of (stable) $(\infty,1)$-categories, the fundamental proof techniques used in the theory of derivators have a strong diagrammatic flavour. Indeed, derivators have been studied by Franke and Garkusha under the name `systems of diagram categories'. Moritz Rahn (n\'e Groth) in \cite{Gro13} has given an excellent exposition of the field and his upcoming monograph \cite{Gro19} should provide an even more expansive account.

However, Rahn's focus (and the focus of most authors) is on the stable/triangulated setting, where all limits and colimits are assumed to exist and satisfy fairly strong regularity properties. These are not the only derivators worth studying; Muro-Raptis in \cite{MurRap17} proved that derivator K-theory could be defined on a derivator merely admitting a zero object on its base and (homotopy) pushouts, a situation analogous to Waldhausen categories. In preparing his doctoral work on derivator K-theory, the author was unable to apply verbatim many existing proofs for the (bicomplete) theory of derivators to this broader class of objects. Many of Rahn's proofs on colimit-preservation or phenomena that felt independent of the existence of right Kan extensions still used the full structure of a derivator.

The first goal of this paper is to give new proofs in many cases in order to develop a more robust theory of \emph{half derivators}, \ie ones representing homotopy theories that may not be bicomplete, but still admit many limits or colimits. We develop the theory from the ground up in full detail to show precisely which results apply to half derivators and which require the strength of a (full) derivator. We unify results written using different conventions (and languages) to give the most streamlined proofs possible. We hope that this thorough accounting makes this paper a resource to those hoping to learn about derivators and those already working in the theory.

In Sections~\ref{sec:prelim} and \ref{sec:CAT}, we set up the 2-categorical context in which we are working. The calculus of mates is an essential proof technique, which we recall in Section~\ref{sec:calculusofmates}. We define (pre)derivators in Section~\ref{sec:preders} with an emphasis on the half derivator story. In Section~\ref{sec:2catofpreders} we define various 2-categories of derivators for examination, and in Section~\ref{sec:homotopyexactsquares} we prove some general results about all half derivators.

The second goal of this paper is to develop in detail the particular theory of left pointed derivators for the purposes of studying derivator K-theory in \cite{Col20a}. The salient features of (bicomplete) pointed derivators are recalled in Section~\ref{sec:pointedders}, and we prove which features persist in half pointed derivators in the final Section~\ref{sec:halfpointed}. We conclude with some final comments about the specific category of half derivators appropriate for studying algebraic K-theory. These results previously appeared together with \cite{Col20a} in the author's doctoral thesis \cite{Col19b}, and all acknowledgements therein apply here. In particular, thanks are due to the UCLA derivator seminar: Paul Balmer, Kevin Arlin (n\'e Carlson), Ioannis Lagkas, and John Zhang. Finally, we thank the anonymous referee for their careful reading and suggestions that improved the flow of this paper. 

\section{2-categorical preliminaries}\label{sec:prelim}

Before giving the definition of a derivator, we will need to set up the categorical context in which we are working. We will assume the reader is familiar with basic definitions in category theory, and if not is directed to~\cite{Mac71} for a good introduction.

For a word on set-theoretic concerns before continuing: recall that a \emph{class} is a collection that may not be a set. It is sometimes beneficial in higher category theory to fix a Grothendieck universe $\sU$ and thereby formalize what it means to be `not a set'. This was first developed in~\cite[Expos\'e~I]{SGA4}, and a more modern discussion can be found in, e.g.,~\cite{Shu08}. Because we will not have to address seriously any issues of size in this paper, we fix the following definitions:
\begin{defn}
A category is called \emph{small} if its class of objects is actually a set. A category is \emph{essentially small} if it is equivalent to a small category. Equivalently, a category is essentially small if the collection of isomorphism classes of its objects is a set.
\end{defn}

This being settled, we can define a 2-category and fix some notation.

\begin{defn}
A \emph{2-category} $\mathbf{C}$ consists of a class of objects, and for any two objects $J,K\in\mathbf{C}$, a category of morphisms $\mathbf{C}(J,K)$. Additionally, for any $I,J,K\in\mathbf C$, the composition map
\begin{equation*}
\mathbf{C}(J,K)\times\mathbf C(I,J)\to\mathbf C(I,K)
\end{equation*}
must be a functor. For each $J\in\mathbf{C}$, there is an identity morphism $\id_J\in\mathbf{C}(J,J)$ which acts as expected with respect to composition. We will use different notation for the category of morphisms depending on the 2-category in question.
\end{defn}
One can think of a 2-category as a category enriched in 1-categories, so that the morphism sets in $\mathbf{C}$ actually have the structure of a (small) category. We have three collections in a 2-category: objects, morphisms, and morphisms-between-morphisms, which we will call \emph{2-morphisms}. The canonical example of a 2-category is the 2-category of 1-categories.

\begin{ex}
Define the 2-category $\Cat$ as follows. The objects are small categories; the morphisms are functors; and the 2-morphisms are natural transformations. By way of general notation, we will denote small categories by $J,K$, or other capital Roman letters. Functors between small categories will be $u,v\colon J\to K$, or other lowercase Roman letters. Natural transformations will be denoted $\alpha\colon u\Rightarrow v$ or other lowercase Greek letters. We will try to reserve $a,b,c,x,y$ for objects of a particular category $K\in\Cat$ and $f\colon x\to y$ for maps in $K$, which we will continue to call `maps' rather than `morphisms'.

We will also consider the 2-category $\CAT$, defined analogously to above but with objects all categories, which may or may not be small. There is a slight set-theoretic issue in that the functor-categories here might fail to be small, but we will we ignore this inconvenience.
\end{ex}

Having defined a 2-category, we can now define maps between 2-categories. These come in a variety of flavors, but we will give the one required for our purposes.
\begin{defn}
Let $\mathbf{C},\mathbf{D}$ be 2-categories. A \emph{strict 2-functor} $\F\colon \mathbf{C}\to\mathbf{D}$ consists of the following data: a functor $\F\colon \mathbf{C}\to\mathbf{D}$ on the underlying 1-categories, and for any $J,K\in\mathbf{C}$, a functor $\F_{J,K}\colon\mathbf{C}(J,K)\to\mathbf{D}(\F(J),\F(K))$. That is, $\F$ respects composition of both 1- and 2-morphisms strictly, not just up to natural isomorphism or some weaker notion.
\end{defn}

We have many competing notions of subcategory in the theory of 2-categories, but the one we will need is closest to the case of 1-categories.
\begin{defn}
A \emph{full sub-2-category} $\mathbf{C}$ of a 2-category $\mathbf{D}$ is a subclass of objects $\mathbf{C}\subset\mathbf{D}$ with the choice of morphism categories $\mathbf{C}(J,K)=\mathbf{D}(J,K)$ for all $J,K\in\mathbf{C}$.
\end{defn}

\begin{ex}
As described above, $\Cat$ is a full sub-2-category of $\CAT$ on the subclass of small categories.
\end{ex}

\section{Particulars in $\CAT$}\label{sec:CAT}

We will need a few constructions which occur inside of $\CAT$ for the axioms of a derivator. In fact, the same constructions may be made within an arbitrary 2-category, but all 2-categories in this paper will be sub-2-categories of $\CAT$, so there is no need to strain one's intuition too much. The generalizations will be left to the interested reader with the help of~\cite[\S 7]{Bor94a}.

A commutative square in the 2-category $\CAT$ has more data than a usual commutative square. We write such squares
\begin{equation}
\vcenter{\xymatrix{
A\ar[r]^-{v}\ar[d]_-{p}&B\ar[d]^-{q}\ar@{}[dl]|\swtrans\ar@{}[dl]<-1.25ex>|\alpha\\
C\ar[r]_-{w}&D
}}
\end{equation}
The data is as follows: four categories $A,B,C,D$; four functors $p,q,v,w$; and a natural transformation $\alpha\colon q\circ v\Rightarrow w\circ p$. Thus the square does not commute per se, but there is a natural map $\alpha_a\colon q(v(a))\to w(p(a))$ in $D$ for every object $a\in A$. Of course, if $\alpha$ is $\id_{qv}$ or a natural isomorphism, then this square commutes in the 1-categorical sense as well as we can expect.

We could also have the natural transformation point the other way, in which case we write
\begin{equation}
\vcenter{\xymatrix{
A\ar[r]^-{v}\ar[d]_-{p}&B\ar[d]^-{q}\ar@{}[dl]|\netrans\ar@{}[dl]<-1.25ex>|\beta\\
C\ar[r]_-{w}&D
}}
\end{equation}
and have for every $a\in A$ a natural map $\beta_a\colon w(p(a))\to q(v(a))$ in $D$. Of course this is the same as the above up to flipping the square over the line $\overline{AD}$, but in the case that we have fixed the data of the outside of the square but are varying the natural transformation, we will write it in this way.

There is a particular type of square that will arise repeatedly in this paper, so we describe it now for future reference.
\begin{defn}
Let $u\colon J\to K$ be any functor, and let $k\in K$ be any object. We define the \emph{comma category} $(u/k)$ as follows: its objects are pairs $j\in J$ with a map~$f\colon u(j)\to k$, and a map $(j,f)\to (j',f')$ in the comma category is a\linebreak map $g\colon j\to j'$ in $J$ making the obvious diagram commute:
\begin{equation}\label{dia:commamorphism}
\vcenter{\xymatrix@C=1em{
u(j)\ar[rr]^-{u(g)}\ar[dr]_(0.4){f}&&u(j')\ar[dl]^(0.4){f'}\\
&k
}}
\end{equation}
One should read the notation $(u/k)$ as `$u$ over $k$', which reminds the reader that maps in $(u/k)$ take place over the identity on $k$.

There is an analogous category $(k/u)$ (`$u$ under $k$') where the objects are instead pairs $j\in J$ with $f\colon k\to u(j)$. 
\end{defn}

These categories fit into a canonical commutative square
\begin{equation}\label{dia:commasquare}
\vcenter{\xymatrix{
(u/k)\ar[r]^-{\pr}\ar[d]_-{\pi_{(u/k)}}&J\ar[d]^-u\ar@{}[dl]|\swtrans\ar@{}[dl]<-1.25ex>|\alpha\\
e\ar[r]_k&K
}}
\end{equation}
Here, $e$ denotes the final category with one object and one (identity) morphism. The functor $k\colon e\to K$ classifies the object $k\in K$. The functor~${\pi_{(u/k)}\colon (u/k)\to e}$ is the unique functor which sends all objects to the only object in $e$ and all maps to the only map in $e$. Finally, $\pr\colon (u/k)\to J$ (read: `projection') is the forgetful functor~$(j,f\colon u(j)\to k)\mapsto j$.

Let $(j,f)\in (u/k)$. The composition around the top and right gives an equality~$u(\pr(j,f))=u(j)$, and the composition around the left and bottom gives an equality~$k(\pi_{(u/k)}(j,f))=k$. Thus the natural transformation $\alpha_{(j,f)}\colon u(j)\to k$ has an obvious candidate: the map $f\colon u(j)\to k$ that was part of the original data of $(j,f)$. If we make this a definition, then $\alpha$ is indeed a natural transformation and we have a commutative square in $\CAT$.

There is an analogous construction for the comma category $(k/u)$:
\begin{equation*}
\vcenter{\xymatrix{
(k/u)\ar[r]^-{\pr}\ar[d]_-{\pi_{(u/k)}}&J\ar[d]^-u\ar@{}[dl]|\netrans\ar@{}[dl]<-1.25ex>|\beta\\
e\ar[r]_k&K
}}
\end{equation*}

The category $(u/k)$ and the data of Diagram~\ref{dia:commasquare} satisfies a universal property in the 2-category $\CAT$: they are a kind of pullback, in the sense that any choice of $T\in\CAT$ with two functors and a natural transformation making the square below commute
\begin{equation*}
\vcenter{\xymatrix{
T\ar[r]\ar[d]&J\ar[d]^-u\ar@{}[dl]|\swtrans\\
e\ar[r]_k&K
}}
\end{equation*}
has a unique factorization through the comma category which respects both the functors and the natural transformations. We therefore name Diagram~\ref{dia:commasquare} an \emph{oriented pullback square}.

We can also generalize the above construction to obtain the oriented pullback of any cospan of categories
\begin{equation}\label{dia:cospan}
\vcenter{\xymatrix{
&J_1\ar[d]^-{u_1}\\
J_2\ar[r]_-{u_2}&K
}}
\end{equation}
defining the comma category $(u_1/u_2)$ or $(u_2/u_1)$. The objects in $(u_1/u_2)$ are\linebreak triples $j_1\in J_1$, $j_2\in J_2$, and a map $f\colon u_1(j_1)\to u_2(j_2)$ in $K$. The maps in $(u_1/u_2)$ come from maps $g_1\colon j_1\to j'_1$ and $g_2\colon j_2\to j'_2$ in $J_1$ and $J_2$ respectively making the requisite square commute:
\begin{equation*}
\vcenter{\xymatrix@C=3em{
u_1(j_1)\ar[r]^-{u_1(g_1)}\ar[d]_-f&u_1(j'_1)\ar[d]^-{f'}\\
u_2(j_2)\ar[r]_-{u_2(g_2)}&u_2(j'_2)
}}
\end{equation*}

\begin{remark}\label{rk:laxpullback}
While we have previously referred to these as \emph{lax pullback squares}, this is not correct. Indeed, according to the nLab `comma objects are often misleadingly called lax pullbacks'. The lax pullback square associated to Diagram~\ref{dia:cospan} is a (weak) 2-limit in the 2-category $\Cat$, which gives a unique (up to equivalence) universal object. Since our comma categories have two non-equivalent orientations, the process of producing them cannot depend only on the data of the cospan. It may be that the oriented pullback square is a sort of limit in a double category, but we do not pursue this line of inquiry further.
\end{remark}

\section{The calculus of mates}\label{sec:calculusofmates}

There is one more phenomenon to explore before defining a derivator. A thorough but older (\ie typewritten and pre-\LaTeX) reference for calculus of mates can be found in~\cite{KelStr74}.

Suppose that we are given the following commutative square in $\CAT$:
\begin{equation}\label{dia:matesquare}
\vcenter{\xymatrix{
J_1\ar[r]^-{v}\ar[d]_-{u_1}&J_2\ar[d]^-{u_2}\ar@{}[dl]|\swtrans\ar@{}[dl]<-1.25ex>|\alpha\\
K_1\ar[r]_-{w}&K_2
}}
\end{equation}
such that $u_1$ and $u_2$ admit right adjoints $r_1$ and $r_2$, respectively. Then we can extend the above picture:
\begin{equation*}
\vcenter{\xymatrix{
K_1\ar@(d,l)[dr]_-{=}\ar[r]^-{r_1}&J_1\ar[r]^-{v}\ar[d]_-{u_1}\ar@{}[dl]|\swtrans\ar@{}[dl]<-1.25ex>|{\varepsilon_1}&J_2\ar[d]_-{u_2}\ar@(r,u)[dr]^-=\ar@{}[dl]|\swtrans\ar@{}[dl]<-1.25ex>|\alpha&{}\ar@{}[dl]|\swtrans\ar@{}[dl]<-1.25ex>|{\eta_2}\\
&K_1\ar[r]_-{w}&K_2\ar[r]_-{r_2}&J_2
}}
\end{equation*}
where $\varepsilon_1$ and $\eta_2$ are the counit and unit of the respective adjunctions. In total, this gives us a natural transformation $v\circ r_1\Rightarrow r_2\circ w$ which we call the \emph{right mate of $\alpha$} and denote $\alpha_\ast$.

Similarly, if we have the other flavor of commutative square and $u_1,u_2$ admit left adjoints $\ell_1,\ell_2$, we obtain
\begin{equation*}
\vcenter{\xymatrix{
K_1\ar@(d,l)[dr]_-{=}\ar[r]^-{\ell_1}&J_1\ar[r]^-{v}\ar[d]_-{u_1}\ar@{}[dl]|\netrans\ar@{}[dl]<-1.5ex>|{\eta_1}&J_2\ar[d]_-{u_2}\ar@(r,u)[dr]^-=\ar@{}[dl]|\netrans\ar@{}[dl]<-1.5ex>|\beta&{}\ar@{}[dl]|\netrans\ar@{}[dl]<-1.5ex>|{\varepsilon_2}\\
&K_1\ar[r]_-{w}&K_2\ar[r]_-{\ell_2}&J_2
}}
\end{equation*}
to construct $\beta_!\colon \ell_2\circ w\Rightarrow v\circ \ell_1$.

In the situation of Diagram~\ref{dia:matesquare}, if $u_1,u_2$ admit right adjoints and $v,w$ admit left adjoints, then it makes sense to talk about both $\alpha_!$ and $\alpha_\ast$. A fair question is the relationship between these two mates, which we will answer shortly.

For one more piece of setup: consider a commutative diagram comprised of two squares
\begin{equation*}
\vcenter{\xymatrix{
J_1\ar[r]^-{v_1}\ar[d]_-{u_1}&J_2\ar[d]_-{u_2}\ar@{}[dl]|\swtrans\ar@{}[dl]<-1.25ex>|{\alpha_1}\ar[r]^-{v_2}&J_3\ar[d]^-{u_3}\ar@{}[dl]|\swtrans\ar@{}[dl]<-1.25ex>|{\alpha_2}\\
K_1\ar[r]_-{w_1}&K_2\ar[r]_-{w_2}&K_3
}}
\end{equation*}
We can take composite of these natural transformations (called their \emph{pasting})
\begin{equation*}
\alpha_2\odot\alpha_1\colon u_3\circ v_2\circ v_1\Rightarrow w_2\circ w_1\circ u_1
\end{equation*}
and take its left/right mate or look at the mates one at a time.

\begin{prop}[\eg Lemma~1.14,~\cite{Gro13}]\label{prop:calculusofmates}\ \\\vspace{-1em}
\begin{enumerate}
\item The calculus of mates is compatible with pasting. That is,\linebreak $(\alpha_2\odot\alpha_1)_!=(\alpha_2)_!\odot(\alpha_1)_!$ and similar for the right mates.
\item The different formations of mates are inverse to each other. That is,\linebreak $\alpha=(\alpha_!)_\ast = (\alpha_\ast)_!$ when $\alpha$ admits both a left and right mate.
\item In the case that $\alpha$ admits both a left and a right mate, then $\alpha_!$ is a natural isomorphism if and only if $\alpha_\ast$ is a natural isomorphism. 
\end{enumerate}
\end{prop}

Once we define the axioms of a derivator, we will use all of these properties in order to study some first consequences of those axioms. As a remark on notation, henceforth we will usually not use $\circ$ when writing a composition of functors and just write $vu$ for $v\circ u$.

\section{Prederivators and derivators}\label{sec:preders}

\begin{defn}
A \emph{prederivator} is a strict 2-functor $\D\colon \Cat\op\to\CAT$, where $\text{op}$ reverses only the 1-morphisms and leaves the 2-morphisms alone.
\end{defn}

By way of notation, for a morphism $u\colon J\to K$ in $\Cat$ we denote by $u^\ast$ the functor $\D(u)\colon \D(K)\to~\D(J)$ in $\CAT$, and for $\alpha\colon u\Rightarrow v$ in $\Cat$ we denote by~$\alpha^\ast$ the natural transformation $\D(\alpha)\colon u^\ast\Rightarrow v^\ast$. Composition is respected strictly, so that $(v u)^\ast = u^\ast v^\ast$ and $(\alpha\odot\beta)^\ast=\alpha^\ast\odot\beta^\ast$. Identities are also preserved, so that $(\id_J)^\ast=\id_{\D(J)}$ and $(\id_u)^\ast=\id_{u^\ast}$.

\begin{ex}\label{ex:prederivators}
For the following examples, let $u\colon J\to K$ be a functor in $\Cat$.
\begin{enumerate}
\item Let $\cC\in\CAT$ be any category. Define the prederivator $\D_\cC$ via the Yoneda embedding:
\begin{equation*}
K\mapsto \Fun(K,\cC),\quad u\colon J\to K\,\mapsto u^\ast\colon \Fun(K,\cC)\to \Fun(J,\cC)
\end{equation*}
with the action on natural transformations obvious. Here the map $u^\ast$ is precomposition with $u$, and for this reason we usually refer to $u^\ast$ as a \emph{pullback functor} even when $\D$ is not this specific prederivator.

\item Let $\cA$ be a Grothendieck abelian category, \eg the category of $R$-modules for any commutative ring $R$. We can then form the (unbounded) derived category $D(\cA)$ without any set-theoretic issues. For any small category $K$, the category $\Fun(K,\cA)$ is still a Grothendieck abelian category, so we may take its derived category as well. The assignment $\D_{\cA}(K)=D(\Fun(K,\cA))$ defines a prederivator $\D_\cA$ with~$u^\ast$ induced by precomposition.

\item Let $\cM$ be a combinatorial model category, \eg the category of simplicial sets with either the standard or Joyal model structure. The first model structure makes \textbf{sSet} a model for topological spaces and the second makes it a model for $(\infty,1)$-categories. In any case, for any small category $K$ the category $\Fun(K,\cM)$ is still a combinatorial model category (using either the injective or projective model structure) where the weak equivalences are defined pointwise. Then define the prederivator $\D_\cM$ by $\D_\cM(K)=\Ho(\Fun(K,\cM))$ with~$u^\ast$ again induced by precomposition.

\item\label{ex:shiftedprederivator} Let $\D$ be any prederivator. Then for any $I\in\Cat$ we may obtain another prederivator $\D^I$ defined by $\D^I(K)=\D(I\times K)$ and $\D^I(u)=\D(\id_I\times u)$. We usually call $\D^I$ a \emph{shifted prederivator}. This is a useful way of constructing new prederivators from old ones and is a key technique in the theory of derivators.
\end{enumerate}
\end{ex}

Recall we are working under the slogan `system of diagram categories', so we would like to define some axioms to ensure that $\D(K)$ looks like $K$-shaped diagrams in some category.

We can immediately identify what category that should be. Let $\D$ be a prederivator, $K$ a small category, and $k\in K$ be any object. Recall that we have a functor that classifies the object $k$ which we denote $k\colon e\to K$. Then for any $X\in \D(K)$, we have an object $k^\ast X\in \D(e)$. Suppose that $f\colon k_1\to k_2$ is a map in $K$. Then we have a corresponding natural transformation $f^\ast\colon k_1^\ast \Rightarrow k_2^\ast$ and thus a map $f^\ast X\colon k_1^\ast X\to k_2^\ast X$ in $\D(e)$. Repeating this process for all objects and maps in $K$, we obtain a functor
\begin{equation*}
\dia_K\colon \D(K)\to \Fun(K,\D(e))
\end{equation*}
which sends $X\in\D(K)$ to the functor which assembles all the above data. We call this an \emph{underlying diagram functor}, and its existence implies that the prederivator $\D$ should be modeling $K$-shaped diagrams in $\D(e)$, which we call the \emph{underlying category} or the \emph{base} of the prederivator. We will refer to the categories $\D(K)$ as \emph{coherent} diagrams, as opposed to the \emph{incoherent} diagrams $\Fun(K,\D(e))$. In analogy with `base', sometimes we will call $\D(K)$ the \emph{levels} of the derivator. 

This motivates the following definition:
\begin{defn}
A \emph{semiderivator} is a prederivator $\D$ satisfying the following two axioms:
\begin{enumerate}
\item[(Der1)] Coproducts are sent to products. Explicitly, consider any set $\{K_a\}_{a\in A}$ of small categories, and let $i_b\colon K_b\to \displaystyle\coprod_{a\in A} K_a$ be the inclusion for any $b\in A$. Pulling back along this inclusion gives a functor
\begin{equation*}
i_b^\ast\colon\D\left(\coprod_{a\in A}K_a\right)\to\D(K_b)
\end{equation*}
which induces a map to the product
\begin{equation*}
(i_b^\ast)_{b\in A}\colon\D\left(\coprod_{a\in A}K_a\right)\to \prod_{b\in A}\D(K_b)
\end{equation*}
We require this map to be an equivalence of categories for any\linebreak collection $\{K_a\}_{a\in A}$.

\item[(Der2)] Isomorphisms are detected pointwise. That is, for any $K\in \Cat$, the underlying diagram functor $\dia_K$ is conservative. More specifically, a map $f\colon X\to Y$ is an isomorphism in $\D(K)$ if and only if the map $k^\ast f\colon k^\ast X\to k^\ast Y$ is an isomorphism for all $k\in K$.
\end{enumerate}
\end{defn}

This is the bare minimum such that $\D$ acts like a system of diagram categories. Der1 means that choosing two disconnected diagrams is the same (up to equivalence) as choosing a diagram on each component, and Der2 says that $\D(e)$ has control over isomorphisms of diagrams. Each of the prederivators in Example~\ref{ex:prederivators} is a semiderivator, as they are precisely constructed to be systems of diagram categories.

Note that both Der1 and Der2 are \emph{properties} of a prederivator $\D$. The functors in these axioms always exist for any prederivator; we are investigating whether they are equivalences or conservative (respectively). Thus we are not placing additional structures on a prederivator to obtain a semiderivator, but examining how nice the 2-functor $\D$ happens to be.

A derivator is a semiderivator that admits homotopy limits and colimits, as well as more general homotopy Kan extensions that may be computed pointwise as homotopy limits and colimits. We will again motivate these as best as possible. Recall that in ordinary category theory, limits and colimits can be thought of as adjoint functors. Specifically, if we let $K$ be a diagram shape and $\cC$ a category, then we can consider the functor $\Delta\colon\cC\to\Fun(K,\cC)$ which sends $c\in \cC$ to the constant diagram of shape~$K$. Assuming that $\cC$ admits limits and colimits of shape $K$, we have the following adjunction:
\begin{equation*}
\vcenter{\xymatrix{
\cC\ar[d]_-\Delta\\
\Fun(K,\cC)\ar@(l,l)[u]^{\colim_K}\ar@(r,r)[u]_{\lim_K}
}}
\end{equation*}
The colimit functor is the left adjoint of $\Delta$, and the limit functor is the right adjoint. We will not recall the construction of left and right Kan extensions in ordinary category theory, but they may be computed pointwise as colimits and limits respectively, see \eg \cite[X.3,~Theorem~1]{Mac98}.

In derivators, we will replace $\cC$ by $\D(e)$ and $\Fun(K,\cC)$ by $\D(K)$. The analogue of~$\Delta$ in this case is the following functor: consider the projection $\pi_K\colon K\to e$ of $K$ to the final category. Then for any $x\in\D(e)$, we have a coherent diagram $\pi_K^\ast x\in \D(K)$. To get a handle on this object, we can restrict to $k\in K$. We then notice that
\begin{equation*}
k^\ast \pi_K^\ast x=(\pi_K k)^\ast x
\end{equation*}
by strict 2-functoriality. The map $\pi_K k\colon e\to e$ must be the identity functor, and thus $(\pi_K k)^\ast=\id_{\D(e)}^\ast$ so $k^\ast\pi_K^\ast x=x$ for any $x\in\D(e)$. Using the same sort of reasoning we can see that for any $k\to k'$ in $K$, the map $k^\ast x\to k'^\ast x$ is the identity on $x$. Therefore~$\pi_K^\ast$ is a coherent constant diagram functor. Having established this, we can now give the axioms of a derivator. Because it will be necessary in the long run, we will give this definition in halves. 

\begin{defn}\label{defn:leftderivator}
A semiderivator $\D$ is a \emph{left derivator} if it satisfies the following two axioms:
\begin{enumerate}
\item[(Der3L)] The semiderivator $\D$ is (homotopically) cocomplete. Specifically, for every functor $u\colon J\to K$, the pullback $u^\ast$ admits a left adjoint, which we denote $u_!\colon\D(J)\to\D(K)$ and call the \emph{(homotopy) left Kan extension along $u$}. As a special case, this includes $\pi_K\colon K\to e$ and thus $\D(e)$ admits all (coherent) colimits.

\item[(Der4L)] Left Kan extensions can be computed pointwise. Let $u\colon J\to K$ and $k\in~K$. Then recall that we have the following oriented pullback square in $\Cat$ from Diagram~\ref{dia:commasquare}
\begin{equation*}
\vcenter{\xymatrix{
(u/k)\ar[r]^-{\pr}\ar[d]_-{\pi}&J\ar[d]^-u\ar@{}[dl]|\swtrans\ar@{}[dl]<-1.25ex>|\alpha\\
e\ar[r]_k&K
}}
\end{equation*}
where we let $\pi=\pi_{(u/k)}$ for brevity. Applying the semiderivator $\D$ to this square, we obtain the following square in $\CAT$, remembering that functors are reversed and natural transformations are not:
\begin{equation*}
\xymatrix{
\D((u/k))&\D(J)\ar@{}[dl]|{\swtrans}\ar@{}[dl]<-1.25ex>|{\alpha^\ast}\ar[l]_-{\pr^\ast}\\
\D(e)\ar[u]^-{\pi^\ast}&\D(K)\ar[l]^-{k^\ast}\ar[u]_-{u^\ast}
}
\end{equation*}
We would like to use techniques in the calculus of mates from Section~\ref{sec:calculusofmates}, but our square does not have the correct orientation. If we flip it around so that the functors point to the bottom-right, we obtain
\begin{equation*}
\vcenter{\xymatrix{
\D(K)\ar[r]^-{k^\ast}\ar[d]_-{u^\ast}&\D(e)\ar[d]^-{\pi^\ast}\ar@{}[dl]|\netrans\ar@{}[dl]<-1.25ex>|{\alpha^\ast}\\
\D(J)\ar[r]_-{\pr^\ast}&\D((u/k))
}}
\end{equation*}
By Der3L, both vertical functors admit left adjoints, so we may construct the left mate of $\alpha^\ast$, which we denote by $\alpha_!$ rather than $(\alpha^\ast)_!$:
\begin{equation*}
\vcenter{\xymatrix{
\D(K)\ar[r]^-{u_!}\ar@(d,l)[dr]_-=&\ar@{}[dl]|\netrans\D(K)\ar[r]^-{k^\ast}\ar[d]_-{u^\ast}&\D(e)\ar[d]^-{\pi^\ast}\ar@{}[dl]|\netrans\ar@{}[dl]<-1.25ex>|{\alpha^\ast}\ar@(r,u)[dr]^-=&{}\ar@{}[dl]|\netrans\\
{}&\D(J)\ar[r]_-{\pr^\ast}&\D((u/k))\ar[r]_-{\pi_!}&\D(e)
}}
\end{equation*}
In total we have the natural transformation $\alpha_!\colon \pi_!\pr^\ast\Rightarrow k^\ast u_!$. We require this map to be a natural isomorphism for all $u\colon J\to K$ and $k\in K$. That this transformation is an isomorphism is called the \emph{Beck-Chevalley condition} in similar contexts.
\end{enumerate}
\end{defn}
Once again, note that these are properties of a semiderivator $\D$, not additional structures.

To analyze Der4L a bit more, for $X\in\D(J)$, $k^\ast u_!X$ is the value of the coherent diagram $u_!X$ at the point $k\in K$. The lefthand side of $\alpha_!$ is the colimit of shape $(u/k)$ of the diagram obtained by pulling back $X$ to the comma category. This is particularly useful in situations where we have better knowledge of the comma category and the projection functor than the functor $u\colon J\to~K$.

For the sake of completeness, we give the dual definition explicitly.
\begin{defn}
A semiderivator $\D$ is a \emph{right derivator} if it satisfies the following two axioms:
\begin{enumerate}
\item[(Der3R)] The semiderivator $\D$ is (homotopically) complete. Specifically, for every functor $u\colon J\to K$, the pullback $u^\ast$ admits a right adjoint, which we denote $u_\ast\colon\D(J)\to\D(K)$ and call the \emph{(homotopy) right Kan extension along $u$}. As a special case, this includes $\pi_K\colon K\to e$ and thus $\D(e)$ admits all (coherent) limits.

\item[(Der4R)] Right Kan extensions can be computed pointwise. Let $u\colon J\to K$ and $k\in K$. Then recall that we have the following oriented pullback square in $\Cat$
\begin{equation*}
\vcenter{\xymatrix{
(k/u)\ar[r]^-{\pr}\ar[d]_-{\pi}&J\ar[d]^-u\ar@{}[dl]|\netrans\ar@{}[dl]<-1.25ex>|\beta\\
e\ar[r]_k&K
}}
\end{equation*}
where we let $\pi=\pi_{(k/u)}$ for brevity. Applying the semiderivator $\D$ to this square and taking the right mate $\beta_\ast=(\beta^\ast)_\ast$ we obtain a natural transformation $\beta_\ast\colon k^\ast u_\ast\Rightarrow \pi_\ast\pr^\ast$. We require this map to be a natural isomorphism for all $u\colon J\to K$ and $k\in K$.
\end{enumerate}
\end{defn}

\begin{remark}\label{rk:halfderivator}
Historically in derivator literature, a derivator which admits all colimits and in which left Kan extensions may be computed pointwise is called a right derivator. This is due to an analogy with right exact functors that we feel does not justify the confusing terminology. We exercise the right to rechristen these objects.
\end{remark}

\begin{defn}
A \emph{derivator} is a semiderivator that is both a left and a right derivator. That is, it is a prederivator satisfying Der1, Der2, Der3L, Der3R, Der4L, and Der4R.
\end{defn}

\begin{ex}\label{ex:derivators}\ \\\vspace{-1em}
\begin{enumerate}
\item For a category $\cC$, the prederivator $\D_\cC$ is a left (resp. right) derivator if and only if $\cC$ is cocomplete (resp. complete).
\item For a Grothendieck abelian category $\cA$ (which is in particular bicomplete), the prederivator $\D_\cA$ is a derivator.
\item For a combinatorial model category $\cM$, working under the assumption that model categories are bicomplete, the prederivator $\D_\cM$ is a derivator. 
\end{enumerate}
\end{ex}

We can also generate new derivators from old:
\begin{prop}[Theorem~1.25,~\cite{Gro13}]\label{prop:shiftedderivatorgood}
Let $\D$ be a left (resp. right) derivator and $I$ a small category. Then the shifted prederivator $\D^I$ of Example~\ref{ex:prederivators}\ref{ex:shiftedprederivator} is also a left (resp. right) derivator. 
\end{prop}

\begin{remark}
There is a `fifth axiom' for derivators which is not necessary in all contexts, ours included. We mention it for the sake of completeness.

Consider the ordinal $[1]$ as a category, that is, the category with two objects $0,1$ and one non-identity map $0\to 1$. For a (pre)derivator $\D$, the category $\D([1])$ is `coherent arrows' in $\D(e)$. Not only can we look at $\dia_{[1]}\colon \D([1])\to\Fun([1],\D(e))$, but also we could do the same for the (pre)derivator $\D^K$ for any small category $K$. This functor is defined to be
\begin{equation*}
\dia_{[1],K}\colon \D^K([1])=\D(K\times[1])\to \Fun([1],\D(K))=\Fun([1],\D^K(e))
\end{equation*}
We call this a \emph{partial underlying diagram functor}; we forget the $[1]$-dimension of coherence but leave the $K$-dimension of the diagram coherent.
\begin{defn}\label{defn:strong}
A prederivator $\D$ is \emph{strong} if it satisfies the following axiom:
\begin{enumerate}
\item[(Der5)] For any category $K\in \Cat$, the functor $\dia_{[1],K}$ is full and essentially surjective.
\end{enumerate}
\end{defn}

Essential surjectivity means that whenever we have a map in $\D(K)$, we can lift it to a coherent arrow between $K$-shaped diagrams, \ie an element of $\D(K\times[1])$. Fullness means that whenever we have a commutative square in $\D(K)$, we can lift it to a map in $\D(K\times[1])$.

This is a condition which usually holds when a prederivator comes from some sort of explicit model; all the derivators of Example~\ref{ex:derivators} are strong. Constructions made inside the theory of derivators will usually preserve the property of being a left/right derivator, but not necessarily of being strong. Lagkas in~\cite{Lag17} gives a heuristic for constructing non-strong derivators arising in the context of monads over a triangulated category translated into derivator theory.

As a final remark, Cisinski in \cite{Cis10} proves that derivators which arise from `cat\'egories d\'erivables' (which include model categories and categories of fibrant objects)  satisfy an even stronger axiom. For notation, a \emph{finite free category} is a finite category with no endomorphisms or relations. For instance, any ordinal category $[n]$ (similar to $[1]$ above) is finite free. Any category whose underlying diagram has a commutative square, however, is not free. All finite free categories have underlying diagrams that look like trees.
\begin{defn}\label{defn:strongstrongness}
Let $I$ be a finite free category. A prederivator is \emph{strong} in the sense of Cisinski if it satisfies the following axiom:
\begin{itemize}
\item[(Der5$'$)] For any category $K\in\Cat$, the functor $\dia_{I,K}\colon \D(K\times I)\to \Fun(I,\D(K))$ is full and essentially surjective.
\end{itemize}
\end{defn}
But as we said above, we are not concerned with strongness for the theory of half derivators, so will not need to distinguish between these two notions.
\end{remark}

Looking at the bases of the derivators in Example~\ref{ex:derivators}, we find a limitation: what if we have some category that does not admit all limits and colimits, but still admits some? What we wanted to produce a derivator modelling diagrams in the bounded derived category $D^b(\cA)$ of an abelian category or the model category of finite CW-complexes? We could easily produce a semiderivator for these cases, but the homotopical bicompleteness on all of $\Cat$ throws us off.

The solution is to restrict our attention at times to a full sub-2-category $\Dia\subset\Cat$ that contains only those diagrams over which we may take a limit or colimit, and replace $\Cat$ by $\Dia$ in all the axioms above. However, not every sub-2-category is appropriate; for example, given any functor $u\colon J\to K$ in $\Dia$ and any $k\in K$, we should also have $(u/k)\in\Dia$ in order to make sense of Der3L. The following axioms were first defined in~\cite[p.3]{Mal07} and refined in~\cite[Definition~1.12]{Gro13}
\begin{defn}\label{defn:diagramcategory}
A full sub-2-category $\Dia\subset\Cat$ is a \emph{diagram 2-category} if it satisfies the following axioms:
\begin{enumerate}
\item[(Dia1)] $\Dia$ contains all finite posets.
\item[(Dia2)] $\Dia$ is closed under finite coproducts and (1-categorical) pullbacks.
\item[(Dia3)] For every $u\colon J\to K$ in $\Dia$ and every $k\in K$, $(u/k),(k/u)\in\Dia$.
\item[(Dia4)] If $K\in\Dia$, then $K\op\in\Dia$.
\item[(Dia5)] For every Grothendieck fibration $u\colon J\to K$ in $\Cat$, if for all $k\in K$ the (strict) fibre~$u\inv(k)$ is in $\Dia$ and $K\in\Dia$, then $J\in\Dia$ as well.
\end{enumerate}
\end{defn}

To motivate the above axioms: Dia1 makes sure that we have something in $\Dia$ to work with; Dia2 makes sure we can check Der1; Dia3 makes sure we can check Der3L and Der3R; Dia4 is to preserve some semblance of duality; and Dia5 is related to the idea that, in this situation, we should be able to `build' $J$ out of the fibres $J_k$ and the base $K$, so it should be a valid diagram shape as well. The smallest choice of diagram 2-category is $\mathbf{Pos_f}$ the 2-category of all finite posets. The largest choice is, of course, $\Cat$ itself.

\begin{remark}\label{rk:dirf}
A common choice for $\Dia$ is $\mathbf{Dir_f}$, the 2-category of finite direct categories, \ie categories whose nerve has finitely many nondegenerate simplices. Keller in~\cite{Kel07} gives a construction of a derivator $\D_\cE$ with domain $\mathbf{Dir_f}$ for any exact category $\cE$ with $\D_\cE(e)=D^b(\cE)$. Cisinski proves that a Waldhausen category $\cW$ whose weak equivalences satisfy some mild properties gives rise to a left derivator $\D_\cW$ with domain $\mathbf{Dir_f}$ such that $\D_\cW(e)=\Ho\cW$ (see Lemma~\ref{lemma:waldhausenderivator}).
 \end{remark}

The main results of the application to K-theory in \cite{Col20a} apply to derivators with domain $\Dirf$. However, to establish the theory of half derivators in abstract, we will allow $\Dia$ arbitrary.

\section{The 2-category of (pre)derivators}\label{sec:2catofpreders}

Having set up the objects of our study, we can now describe the morphisms between them and  the 2-morphisms between those.
\begin{defn}
Let $\D,\E\colon\Dia\op\to\CAT$ be prederivators. A \emph{morphism of prederivators} $\Phi\colon\D\to\E$ is a pseudonatural transformation of the associated 2-functors. This consists of the following data: for each $K\in\Dia$ we have a\linebreak functor $\Phi_K\colon\D(K)\to~\E(K)$ and for every $u\colon J\to K$ we have a natural isomorphism $\gamma_u^\Phi\colon u^\ast \Phi_K\Rightarrow \Phi_J u^\ast$
\begin{equation*}
\xymatrix{
\D(K)\ar[r]^{\Phi_K}\ar[d]_-{u^\ast}&\E(K)\ar[d]^-{u^\ast}\ar@{}[dl]|{\swtrans}\ar@{}@<-1.75ex>[dl]|{\gamma_u^\Phi}\\
\D(J)\ar[r]_{\Phi_J}&\E(J)
}
\end{equation*}
where we have slightly abused notation by writing $u^\ast$ for both $\D(u)$ and $\E(u)$.

The family $\{\gamma_u^\Phi\}$ is subject to coherence conditions. Foremost, for two composable functors $u\colon J\to K$ and $v\colon I\to J$ we require that the pasting on the left be equal to the square on the right:
\begin{equation*}
\vcenter{\xymatrix{
\D(K)\ar[r]^{\Phi_K}\ar[d]_-{u^\ast}&\E(K)\ar[d]^-{u^\ast}\ar@{}[dl]|{\swtrans}\ar@{}@<-1.75ex>[dl]|{\gamma_u^\Phi}\\
\D(J)\ar[d]_-{v^\ast}\ar[r]^{\Phi_J}&\E(J)\ar[d]^-{v^\ast}\ar@{}[dl]|{\swtrans}\ar@{}@<-1.75ex>[dl]|{\gamma_v^\Phi}\\
\D(I)\ar[r]_{\Phi_I}&\E(I)
}}
\quad=\quad
\vcenter{\xymatrix{
\D(K)\ar[r]^{\Phi_K}\ar[d]_-{(uv)^\ast}&\E(K)\ar[d]^-{(uv)^\ast}\ar@{}[dl]|{\swtrans}\ar@{}@<-1.75ex>[dl]|{\gamma_{uv}^\Phi}\\
\D(I)\ar[r]_{\Phi_I}&\E(I)
}}
\end{equation*}
In addition, we require $\gamma_{\id_J}^\Phi=\id_{\D(J)}$. There is also required compatibility  with natural transformations in $\Dia$. For two functors $u,v\colon J\to K$ and a natural transformation $\alpha\colon u\Rightarrow~v$, we require the below pastings to be equal:
\begin{equation*}
\vcenter{\xymatrix{
\D(K)\ar[r]^{\Phi_K}\ar[d]^-{u^\ast}="S0"\ar@(l,l)[d]_-{v^\ast}="T0"&\E(K)\ar[d]^-{u^\ast}\ar@{}[dl]|{\swtrans}\ar@{}@<-1.75ex>[dl]|{\gamma_u^\Phi}\\
\D(J)\ar[r]_{\Phi_J}&\E(J)
\ar@{}"S0";"T0" |{\mathbin{\rotatebox[origin=c]{180}{$\Rightarrow$}}}_-{\alpha^\ast}
}}
\quad=\quad
\vcenter{\xymatrix{
\D(K)\ar[r]^{\Phi_K}\ar[d]_-{v^\ast}&\E(K)\ar[d]_-{v^\ast}="T0"\ar@(r,r)[d]^-{u^\ast}="S0"\ar@{}[dl]|{\swtrans}\ar@{}@<-1.75ex>[dl]|{\gamma_v^\Phi}\\
\D(J)\ar[r]_{\Phi_J}&\E(J)
\ar@{}"S0";"T0" |{\mathbin{\rotatebox[origin=c]{180}{$\Rightarrow$}}}_-{\alpha^\ast}
}}
\end{equation*}
The general definition of a pseudonatural transformation can be found at~\cite[Definition~7.5.2]{Bor94a}.

A \emph{morphism of (left/right) derivators} is a just a morphism of prederivators; there is no additional condition. A morphism $\Phi\colon\D\to\E$ is an \emph{equivalence of (pre)derivators} if $\Phi_K$ is an equivalence of categories for every $K\in\Dia$.
\end{defn}

Having claimed that we are assembling a 2-category, we now define the 2-morphisms.

\begin{defn}
If $\Phi,\Psi\colon\D\to \E$ are two morphisms of (pre)derivators, a 2-morphism \linebreak$\mu\colon\Phi\to\Psi$ is given by a \emph{modification} of pseudonatural transformations of 2-functors. This is a natural transformation $\mu_K\colon\Phi_K\Rightarrow\Psi_K$ for every $K\in\Dia$ satisfying the following coherence condition: if $u,v\colon J\to K$ are two functors and $\alpha\colon u\Rightarrow v$ is a natural transformation, then we have an equality of pastings
\begin{equation*}
\vcenter{\xymatrix@C=3em{
\D(K)\ar@/^1pc/[r]^-{u^\ast}="S1"\ar@/_1pc/[r]_-{v^\ast}="T1"&\D(J)\ar@/^1pc/[r]^-{\Phi_J}="S0"\ar@/_1pc/[r]_-{\Psi_J}="T0"&\E(J)
\ar@{} "S0";"T0" |{\mathbin{\rotatebox[origin=c]{270}{$\Rightarrow$}}} \ar@{}@<-1.4ex>"S0";"T0"|{\mu_J}
\ar@{} "S1";"T1" |{\mathbin{\rotatebox[origin=c]{270}{$\Rightarrow$}}} \ar@{}@<-1.4ex>"S1";"T1"|{\alpha^\ast}
}}
\quad=\quad
\vcenter{\xymatrix@C=3em{
\D(K)\ar@/^1pc/[r]^-{\Phi_K}="S0"\ar@/_1pc/[r]_-{\Psi_K}="T0"&\E(K)\ar@/^1pc/[r]^-{u^\ast}="S1" \ar@/_1pc/[r]_-{v^\ast}="T1"&\E(J)
\ar@{} "S0";"T0" |{\mathbin{\rotatebox[origin=c]{270}{$\Rightarrow$}}} \ar@{}@<-1.4ex>"S0";"T0"|{\mu_K}
\ar@{} "S1";"T1" |{\mathbin{\rotatebox[origin=c]{270}{$\Rightarrow$}}} \ar@{}@<-1.4ex>"S1";"T1"|{\alpha^\ast}
}}
\end{equation*}
See also~\cite[Definition~7.5.3]{Bor94a}. A modification $\mu$ is called an \emph{isomodification} if $\mu_K$ is a natural isomorphism for every $K\in\Dia$.
\end{defn}

This gives us a 2-category $\PDer$ of prederivators and full sub-2-categories left derivators, right derivators, and derivators. We name this last one $\Der$. By way of notation, we will reserve uppercase Greek letters $\Phi,\Psi$ for morphisms of derivators, lowercase Greek letters $\mu,\nu$ for modifications.

On the surface, these definitions require a ton of compatible information, and it seems unlikely that we would ever be able to construct morphisms of derivators. There is a ready source of morphisms, however, which arise from functors in $\Dia$. Consider a functor $u\colon J\to K$ and the associated pullback $u^\ast\colon\D(K)\to\D(J)$. Then we can consider $u^\ast\colon\D^K\to\D^J$ as a morphism between the associated shifted derivators. Let us be explicit about where the structure isomorphisms $\gamma^{u^\ast}$ come from in this case. Let $v\colon A\to B$ be a functor in $\Dia$. Then we need to populate the below square with a natural isomorphism:
\begin{equation*}
\vcenter{\xymatrix{
\D^K(B)\ar[r]^-{u^\ast_B}\ar[d]_-{v^\ast}&\D^J(B)\ar[d]^-{v^\ast}\ar@{}[dl]|{\swtrans}\ar@{}@<-1.75ex>[dl]|{\gamma_v^{u^\ast}}\\
\D^K(A)\ar[r]_-{u^\ast_A}&\D^J(A)
}}
\end{equation*}
If we make more explicit what all these maps are in terms of the derivator $\D$, the above is equal to 
\begin{equation}\label{eq:pullbackasmorphism}
\vcenter{\xymatrix@C=3em{
\D(K\times B)\ar[r]^-{(u\times\id_B)^\ast}\ar[d]_-{(\id_K\times v)^\ast}&\D(J\times B)\ar[d]^-{(\id_J\times v)^\ast}\ar@{}[dl]|{\swtrans}\ar@{}@<-1.75ex>[dl]|{\gamma_v^{u^\ast}}\\
\D(K\times A)\ar[r]_-{(u\times \id_A)^\ast}&\D(J\times A)
}}
\end{equation}
So we are looking for a transformation $(\id_J\times v)^\ast(u\times\id_B)^\ast\Rightarrow(u\times\id_A)^\ast(\id_K\times v)^\ast$. But by strict 2-functoriality of $\D$, these are the same functor $(u\times v)^\ast$, so the\linebreak choice $\gamma_v^{u^\ast}=\id_{(u\times v)^\ast}$ fits the bill. This choice happily satisfies all compatibility conditions. In fact, we have terminology for just this situation.

\begin{defn}
A morphism of derivators $\Phi\colon\D\to\E$ is called \emph{strict} if for\linebreak any $u\colon J\to~K$ in $\Dia$, the corresponding structure natural isomorphism $\gamma^{\Phi}_u$ is the identity.
\end{defn}

Another class of morphisms of derivators arise from left and right Kan extensions, but these are not strict. For any $u\colon J\to K$ in $\Dia$, if $\D$ is a left derivator then we have a morphism $u_!\colon\D^J\to\D^K$. For $v\colon A\to B$, the structure isomorphisms in this case need to populate the square
\begin{equation*}
\vcenter{\xymatrix@C=3em{
\D(J\times B)\ar[r]^-{(u\times\id_B)_!}\ar[d]_-{(\id_J\times v)^\ast}&\D(K\times B)\ar[d]^-{(\id_K\times v)^\ast}\ar@{}[dl]|{\swtrans}\ar@{}@<-1.75ex>[dl]|{\gamma_v^{u_!}}\\
\D(J\times A)\ar[r]_-{(u\times \id_A)_!}&\D(K\times A)
}}
\end{equation*}
The natural transformation above is the left mate of Diagram~\ref{eq:pullbackasmorphism}, and it is not hard to show that this mate is a natural isomorphism. Instead of showing it directly, we will show shortly how to make this conclusion. Unfortunately, there is no reason for this mate to be the identity, which corresponds to the fact that colimits are unique up to unique isomorphism, but not strictly unique.

\begin{defn}
Let $\D,\E$ be left derivators and $u\colon J\to K$ in $\Dia$. We say that a morphism $\Phi\colon\D\to\E$ \emph{preserves left Kan extensions along $u$} if the left mate of $(\gamma^\Phi_u)\inv$ is a natural isomorphism. Specifically, we have the pasting
\begin{equation*}
\xymatrix{
\D(J)\ar@(d,l)[dr]_-=\ar[r]^-{u_!}&\D(K)\ar@{}[dl]|{\netrans}\ar[r]^{\Phi_K}\ar[d]_-{u^\ast}&\E(K)\ar[d]^-{u^\ast}\ar@{}[dl]|{\netrans}\ar@{}@<-1.75ex>[dl]|{(\gamma_u^\Phi)\inv}\ar@(r,u)[dr]^-=&\ar@{}[dl]|{\netrans}\\
{}&\D(J)\ar[r]_{\Phi_J}&\E(J)\ar[r]_{u_!}&\E(K)
}
\end{equation*}
giving us a natural transformation $(\gamma_u^\Phi)\inv_!\colon u_!\Phi_J\Rightarrow\Phi_K u_!$ which we demand is an isomorphism, where again we slightly abuse notation by writing $u_!$ for the left adjoint to both $\D(u)$ and $\E(u)$. If the morphism $\Phi$ preserves left Kan extensions along\linebreak all $u\colon J\to K$ in $\Dia$, we say that $\Phi$ is \emph{cocontinuous}.
\end{defn}

Cocontinuity means that we can compute the left Kan extension along $u$ in $\D$ then apply $\Phi$, or apply $\Phi$ and compute the left Kan extension along $u$ in $\E$ and we obtain isomorphic objects. There is an analogous definition of a \emph{continuous} morphism that we will spell out explicitly for reference.
\begin{defn}
Let $\D,\E$ be right derivators and $u\colon J\to K$ in $\Dia$. We say that a morphism $\Phi\colon\D\to\E$ \emph{preserves right Kan extensions along $u$} if the right mate of $\gamma^\Phi_u$ is a natural isomorphism. Specifically, we have the pasting
\begin{equation*}
\xymatrix{
\D(J)\ar@(d,l)[dr]_-=\ar[r]^-{u_\ast}&\D(K)\ar@{}[dl]|{\swtrans}\ar[r]^{\Phi_K}\ar[d]_-{u^\ast}&\E(K)\ar[d]^-{u^\ast}\ar@{}[dl]|{\swtrans}\ar@{}@<-1.75ex>[dl]|{\gamma_u^\Phi}\ar@(r,u)[dr]^-=&\ar@{}[dl]|{\swtrans}\\
{}&\D(J)\ar[r]_{\Phi_J}&\E(J)\ar[r]_{u_\ast}&\E(K)
}
\end{equation*}
giving us a natural transformation $(\gamma_u^\Phi)_\ast\colon \Phi_K u_\ast\Rightarrow u_\ast\Phi_J$. If the morphism $\Phi$ preserves right Kan extensions along all $u\colon J\to K$ in $\Dia$, we say that $\Phi$ is \emph{continuous}.
\end{defn}

\begin{remark}
It might seem like Proposition~\ref{prop:calculusofmates}(3) says that a morphism of derivators~$\Phi$ is cocontinuous if and only if it is continuous. However, cocontinuity uses $(\gamma^u_\Phi)\inv$ while continuity uses $\gamma^u_\Phi$, and the calculus of mates says nothing about the relationship between the mates of $\alpha$ and $\alpha\inv$.
\end{remark}

We will now focus on cocontinuity, leaving the dual formulations to the reader. Let $\Phi\colon\D\to\E$ be a morphism between left derivators. We will be able to determine if $\Phi$ is cocontinuous in two key ways.

\begin{prop}[Proposition~2.3, \cite{Gro13}]\label{prop:colimitpreserving}
A morphism $\Phi\colon\D\to\E$ between left derivators is cocontinuous if and only if $\Phi$ preserves left Kan extensions along all\linebreak maps $\pi\colon K\to e$ for $K\in\Dia$, \ie $\Phi$ preserves (homotopy) colimits.
\end{prop}
\begin{proof}
We begin with some observations. Der2 tells us that isomorphisms may be checked pointwise, so that for a particular $X\in\D(J)$, we have that
\begin{equation*}
(\gamma_u^\Phi)\inv_{!,X}\colon u_!\Phi_JX\to\Phi_K u_! X
\end{equation*}
if an isomorphism if and only if
\begin{equation*}
k^\ast(\gamma_u^\Phi)\inv_{!,X}\colon k^\ast u_!\Phi_JX\to k^\ast\Phi_K u_! X
\end{equation*}
is an isomorphism for all $k\in K$. This tactic will be common, so we give it a name: precomposing or postcomposing a natural transformation by a functor is generally called \emph{whiskering}. We will now modify the domain and codomain of this map (up to isomorphism) to make it easier to study.

We may postcompose with the structure isomorphism $\gamma_k^\Phi\colon k^\ast\Phi_K\Rightarrow \Phi_e k^\ast$ to obtain
\begin{equation*}
\vcenter{\xymatrix@C=4em{
k^\ast u_!\Phi_JX\ar[r]^-{k^\ast(\gamma_u^\Phi)\inv_{!,X}}& k^\ast\Phi_K u_! X\ar[r]_-\cong^-{(\gamma_k^\Phi)_{u_! X}}\ar[r]&\Phi_e k^\ast u_! X
}}
\end{equation*}
Now using Der4L, we can precompose both the domain and codomain by the natural isomorphism $\pi_!\pr^\ast\Rightarrow k^\ast u_!$, recalling the notation from Definition~\ref{defn:leftderivator}. This gives us
\begin{equation*}
\vcenter{\xymatrix@C=4em{
k^\ast u_!\Phi_JX\ar[r]^-{k^\ast(\gamma_u^\Phi)\inv_{!,X}}& k^\ast\Phi_K u_! X\ar[r]_-\cong^-{(\gamma_k^\Phi)_{u_! X}}\ar[r]&\Phi_e k^\ast u_! X\\
\pi_!\pr^\ast\Phi_J X\ar[u]^-\cong&&\Phi_e\pi_!\pr^\ast X\ar[u]_-\cong
}}
\end{equation*}
As one final step, we can commute $\pr^\ast$ and $\Phi$ using $\gamma_{\pr}^\Phi$ and complete the above to a commutative square
\begin{equation*}
\vcenter{\xymatrix@C=4em{
k^\ast u_!\Phi_JX\ar[r]^-{k^\ast(\gamma_u^\Phi)\inv_{!,X}}& k^\ast\Phi_K u_! X\ar[r]_-\cong^-{(\gamma_k^\Phi)_{u_! X}}\ar[r]&\Phi_e k^\ast u_! X\\
\pi_!\pr^\ast\Phi_J X\ar[u]^-\cong\ar[r]_-{\pi_!(\gamma_{\pr}^\Phi)_X}^-\cong&\pi_!\Phi_{(u/k)}\pr^\ast X\ar[r]_-{(\gamma_\pi^\Phi)\inv_{!,\pr^\ast X}}&\Phi_e\pi_!\pr^\ast X\ar[u]_-\cong
}}
\end{equation*}
We are able to fill in the bottom-right map because of the functoriality of the calculus of mates and the coherence conditions imposed on the structure isomorphisms $\{\gamma^\Phi\}$. We have thus reduced the question of all $(\gamma_u^\Phi)\inv_!$ being natural isomorphisms to the specific case of $(\gamma_\pi^\Phi)\inv_!$ being a natural isomorphism for all maps $\pi\colon (u/k)\to e$, at least on objects of the form $\pr^\ast X$. This is a key technique in derivator proofs: we have reduced a general problem to more specific one in the base of the derivator. This reduction allows us to conclude the lemma.
\end{proof}
\begin{remark}
Rahn's proof of the preceding lemma is identical to ours, but he uses purely the calculus of mates and pasting. We spell out the technique of whiskering and the use of Der2 explicitly in anticipation of Lemma~\ref{lemma:horzhomotopyexactcancel}.
\end{remark}

For the second criterion for cocontinuous morphisms, we recall the following from 1-category theory: left adjoint functors preserve colimits and right adjoint functors preserve limits. Since preserving colimits is enough to preserve all left Kan extensions, we need to figure out what an adjunction of morphisms of derivators should be. There is a general definition in any 2-category that applies here:
\begin{defn}
Let $\Phi\colon\D\to\E$ and $\Psi\colon \E\to\D$ be two morphisms of (pre)derivators. We say that \emph{$\Phi$ is left adjoint to $\Psi$} (equivalently, \emph{$\Psi$ is right adjoint to $\Phi$}) if there exist two modifications $\eta\colon \id_\D\Rightarrow\Psi\Phi$ and $\varepsilon\colon \Phi\Psi\Rightarrow\id_\E$ satisfying the usual triangle identities.
\end{defn}

In particular, an adjunction $(\Phi,\Psi)$ gives rise to an adjunction of functors $(\Phi_K,\Psi_K)$ for each $K\in\Dia$. However, this condition is not sufficient. A morphism of\linebreak derivators $\Phi\colon\D\to\E$ may admit a right adjoint to $\Phi_K\colon\D(K)\to\E(K)$ for\linebreak all $K\in\Dia$, but part of the data of a right adjoint morphism of derivators is the structure isomorphisms, which we have no way of recovering in this general situation.

\begin{prop}[Proposition~2.9, \cite{Gro13}]
Let $\Phi\colon\D\to\E$ be a morphism of left derivators such that each $\Phi_K$ admits a right adjoint $\Psi_K$. Then the collection $\{\Psi_K\}$ assemble to a morphism of derivators $\Psi\colon\E\to\D$ which is right adjoint to $\Phi$ if and only if $\Phi$ is cocontinuous.
\end{prop}
\begin{proof}
Let us begin with the backwards direction. To obtain a morphism of derivators $\Psi\colon\E\to\D$, for all $u\colon J\to K$, we need to come up with a structure isomorphism in the following square:
\begin{equation*}
\xymatrix{
\E(K)\ar[r]^-{\Psi_K}\ar[d]_-{u^\ast}&\D(K)\ar[d]^-{u^\ast}\ar@{}[dl]|{\swtrans}\ar@{}@<-1.75ex>[dl]|{\gamma_u^\Psi}\\
\E(J)\ar[r]_-{\Psi_J}&\D(J)
}
\end{equation*}
Using the fact that $(\Phi_K,\Psi_K)$ and $(\Phi_J,\Psi_J)$ are adjunctions, we may take the left mate of this square (after flipping it for convenience):
\begin{equation*}
\xymatrix{
\D(K)\ar[r]^-{\Phi_K}\ar@(d,l)[dr]_-=&\E(K)\ar[d]_-{\Psi_K}\ar@{}[dl]|{\netrans}\ar[r]^-{u^\ast}&\E(J)\ar[d]^-{\Psi_J}\ar@{}[dl]|{\netrans}\ar@{}@<-1.75ex>[dl]|{\gamma_u^\Psi}\ar@(r,u)[dr]^-=&{}\ar@{}[dl]|{\netrans}\\
{}&\D(K)\ar[r]_-{u^\ast}&\D(J)\ar[r]_-{\Phi_J}&\E(J)
}
\end{equation*}
This gives us a transformation $\Phi_J u^\ast\Rightarrow u^\ast\Phi_K$. We are already equipped with a candidate transformation here, namely $(\gamma_u^\Phi)\inv$. Thus we have the notion\linebreak that $(\gamma_u^\Psi)_! = (\gamma_u^\Phi)\inv$. Using Proposition~\ref{prop:calculusofmates}(2), we may take the right mate of both these transformations and conclude that the natural choice for structure isomorphisms is $\gamma_u^\Psi=((\gamma_u^\Psi)_!)_\ast=(\gamma_u^\Phi)\inv_\ast$.

Unfortunately, we do not know that $\gamma_u^\Psi$ defined this way is an isomorphism. The left and right mates of $\gamma_u^\Phi$ have no particular properties for an arbitrary morphism~$\Phi$. However, if $\Phi$ is cocontinuous, then $(\gamma_u^\Phi)\inv_!$ are natural isomorphisms for all $u\colon J\to K$ in $\Dia$. By Proposition~\ref{prop:calculusofmates}(3), this implies that $(\gamma_u^\Phi)\inv_\ast$ are natural isomorphisms as well, which allows us to furnish the collection $\{\Psi_K\}$ with structure isomorphisms $\{\gamma_u^\Psi\}$.

On the other hand, if $\Psi$ is a right adjoint morphism of derivators, then we have an equality $(\gamma_u^\Psi)\inv=(\gamma_u^\Phi)_\ast$ for any $u\colon J\to K$ using the same calculus of mates as above. Thus $(\gamma_u^\Phi)_\ast$ and $(\gamma_u^\Phi)_!$ are isomorphisms, proving that $\Phi$ is cocontinuous.
\end{proof}

Using these two propositions, we can give a number of cocontinuous morphisms of derivators.

\begin{ex}
Let $u\colon J\to K$ be a functor in $\Dia$.
\begin{enumerate}
\item If $u$ admits a categorical right adjoint $v\colon K\to J$, then $u^\ast\colon\D(K)\to\D(J)$ is right adjoint to $v^\ast\colon\D(J)\to\D(K)$ because (strict) 2-functors send adjunctions to adjunctions, though in our case which is left and which is right swaps. We can upgrade this to, for any prederivator $\D$, a left adjoint morphism $v^\ast\colon \D^K\to\D^J$ which preserves any left Kan extensions that $\D^K$ happens to have.

\item If $\D$ is a left derivator, then the left adjoint functor $u_!\colon \D(J)\to\D(K)$ lifts to a left adjoint morphism of derivators $u_!\colon \D^J\to\D^K$ with right adjoint $u^\ast$. Similarly, if $\D$ is a right derivator, $u^\ast\colon \D^K\to\D^J$ is a left adjoint morphism of derivators.

\end{enumerate}
\end{ex}

\section{Exact squares}\label{sec:homotopyexactsquares}  

At this point, we will begin to use more seriously the calculus of mates to prove a few statements that apply broadly to any derivator $\D$ on any diagram 2-category $\Dia$. The idea is the following: the axiom Der4 has given us a class of squares in $\Dia$ whose left or right mates will always be natural isomorphisms. On the other hand, Proposition~\ref{prop:calculusofmates} gives us a way to compare the mates of pastings with pastings of mates. Therefore we might be able to conclude that some other squares automatically have this same property.

We remark at this point that Maltsiniotis in \cite{Mal12} proved versions of some results below, although the author was not previously aware of his work. Just as our interpreation of Rahn's results, however, we are able to improve some of Maltsiniotis' results and clear up differences in notation. As above, we will cite the relevant results as they arise.

Let us give the first definition.
\begin{defn}\label{defn:homotopyexactsquare}
Let $\D$ be a derivator. Consider the following square in $\Dia$:
\begin{equation}\label{dia:homotopyexactsquare}
\vcenter{\xymatrix{
A\ar[r]^-{v}\ar[d]_-{p}&B\ar[d]^-{q}\ar@{}[dl]|\swtrans\ar@{}[dl]<-1.25ex>|\alpha\\
C\ar[r]_-{w}&D
}}
\end{equation}
We call such a square \emph{$\D$-exact} if the right mate of $\D(\alpha)=\alpha^\ast$,\linebreak namely $\alpha_\ast\colon q^\ast w_\ast\Rightarrow v_\ast p^\ast$, is a natural isomorphism. Equivalently, by Proposition~\ref{prop:calculusofmates}(3), we could ask that the left mate $\alpha_!\colon p_! v^\ast\Rightarrow w^\ast q_!$ be a natural isomorphism. 
\end{defn}

\begin{remark}
Technically our notation for the left and right mates of $\D(\alpha)=\alpha^\ast$ should be $(\alpha^\ast)_!$ and $(\alpha^\ast)_\ast$, but since we will never be taking the mates of the transformation $\alpha\colon qv\Rightarrow wp$ in $\Dia$ we remove the ${}^\ast$ from the notation.
\end{remark}

\begin{remark}\label{rk:homotopyexact}
The usual terminology for a square as in Diagram~\ref{dia:homotopyexactsquare} which is $\D$-exact for every derivator $\D$ is \emph{homotopy exact}. However, we shy away from using this terminology for a few reasons.

It still makes sense to ask if such a square is $\D$-exact for a \emph{prederivator} $\D$ in the case that $\D$ admits the appropriate left or right Kan extensions. In particular, any right derivator admits the right mate $\alpha_\ast$ and any left derivator admits the left mate~$\alpha_!$. Thus we might say Diagram~\ref{dia:homotopyexactsquare} is homotopy exact if it is $\D$-exact for every \emph{half} derivator $\D$.

Work by Cisinski in \cite{Cis06} and \cite{Cis08} provides the strongest connection between actual homotopy theory, \ie the homotopy theory of spaces, and derivators. This is one of the sources of the family of terms `homotopy X' meaning `$\D$-X for any derivator $\D$'. But Cisinki's results depend very strongly on the choice of $\Dia=\Cat$ and his earlier work on basic localisers in $\Cat$.

Maltsiniotis' approach to exact squares has much in common with Cisinski's methods. Many of Maltsiniotis' proofs rely on the fact that, if $\D$ is a derivator on $\Cat$, then we can obtain a basic localiser in $\Cat$ of $\D$-equivalences, see \cite[\S4.5]{Mal12}. General results on basic localisers, \eg \cite[Lemme~3.12]{Mal12}, then apply automatically to the theory of derivators. Our proofs will work more with derivator technology per se.

Unfortunately, there does not seem to be any guarantee that analogous results on basic localisers hold for the case of derivators on $\Dia\subsetneq\Cat$ or half derivators defined on $\Cat$. We hope to address these two cases in future work.
\end{remark}

By the axioms we have put on derivators, we already know that, for any (left and right) derivator~$\D$, the following squares are $\D$-exact for all $u\colon J\to K$ and all $k\in K$:
\begin{equation*}
\vcenter{\xymatrix{
(u/k)\ar[r]^-{\pr}\ar[d]_-{\pi}&J\ar[d]^-u\ar@{}[dl]|\swtrans\ar@{}[dl]<-1.25ex>|\alpha\\
e\ar[r]_k&K
}}
\qquad
\vcenter{\xymatrix{
(k/u)\ar[r]^-{\pr}\ar[d]_-{\pi}&J\ar[d]^-u\ar@{}[dl]|\netrans\ar@{}[dl]<-1.25ex>|\beta\\
e\ar[r]_k&K
}}
\end{equation*}

We have a second class of examples coming from adjunctions in $\Dia$.

\begin{prop}[Proposition~1.18,~\cite{Gro13}]\label{prop:adjointhomotopyexact}
Let $\ell\colon A\to B$ be a left adjoint functor and let $\D$ be any \emph{prederivator}. Then the following commutative square is $\D$-exact:
\begin{equation*}
\vcenter{\xymatrix{
A\ar[r]^-\ell\ar[d]_-{\pi_A}&B\ar[d]^-{\pi_B}\ar@{}[dl]|\netrans\ar@{}[dl]<-1.75ex>|\id\\
e\ar[r]_-{\id_e}&e
}}
\end{equation*}

Dually, for any right adjoint functor $r\colon A\to B$, the following square is $\D$-exact:
\begin{equation*}
\vcenter{\xymatrix{
A\ar[r]^-r\ar[d]_-{\pi_A}&B\ar[d]^-{\pi_B}\ar@{}[dl]|\swtrans\ar@{}[dl]<-1.25ex>|\id\\
e\ar[r]_-{\id_e}&e
}}
\end{equation*}
\end{prop}
\begin{proof}
We will prove the second case, with the modifications for the first obvious. We make use of the following fact: any strict 2-functor sends adjunctions to adjunctions. The data of an adjunction is a pair of functors with a unit and counit satisfying the triangle identities. Specifically for our case: $r\colon A\to B$, $\ell\colon B\to A$, and $\eta\colon \id_A\Rightarrow r\ell$, $\varepsilon\colon \ell r\Rightarrow \id_B$ such that $\varepsilon_\ell\ell\eta=\id_\ell$ and $r\varepsilon\eta_r=\id_r$. Applying any prederivator $\D$ gives us $r^\ast\colon\D(B)\to\D(A)$, $\ell^\ast\colon\D(A)\to\D(B)$, and
\begin{equation*}
\eta^\ast\colon\id_{\D(A)}=(\id_A)^\ast\Rightarrow (r\ell)^\ast=\ell^\ast r^\ast,\quad \varepsilon^\ast\colon r^\ast \ell^\ast=(\ell r)^\ast\Rightarrow (\id_B)^\ast=\id_{\D(B)}
\end{equation*}
The equalities above are a consequence of the strict 2-functoriality\linebreak of $\D\colon\Dia\op\to\CAT$, and that the triangle equalities still hold is another consequence.  However, the composition of the left and right adjoint now seem to be backwards, but this is because our domain was $\Dia\op$ to begin with; the order of the adjunction is reversed. Thus the adjunction $(\ell,r)$ in $\Dia$ yields an adjunction~$(r^\ast,\ell^\ast)$ in $\CAT$ for any prederivator $\D$.

Now, consider the image of our square after applying a derivator $\D$:
\begin{equation*}
\vcenter{\xymatrix{
\D(A)&\D(B)\ar[l]_-{r^\ast}\ar@{}[dl]|\swtrans\ar@{}[dl]<-1.25ex>|{\id^\ast}\\
\D(e)\ar[u]^-{\pi_A^\ast}&\D(e)\ar[l]^-{\id_e^\ast}\ar[u]_-{\pi_B^\ast}
}}
\end{equation*}
Because $r^\ast$ is now a left adjoint, we can take the right mate of this square to obtain~$\id_\ast\colon \id_{e,\ast}\pi_B^\ast\Rightarrow \ell^\ast\pi_A^\ast$, where we write $\id_{e,\ast}$ for illustration but note that it is just the identity on $\D(e)$ because $\id_e^\ast$ is. But now, $\ell^\ast \pi_A^\ast=(\pi_A \ell)^\ast$ by strict 2-functoriality, and $\pi_A \ell\colon B\to e$ must be equal to $\pi_B\colon B\to e$ as $e$ is the final category. The properties of the calculus of mates expressed in Proposition~\ref{prop:calculusofmates} implies that $\id_\ast$ expresses this equality, and thus we conclude that the square is $\D$-exact. 
\end{proof}

Under Maltsiniotis' conditions, this is the claim that axioms CI 1$'$d/g of \cite[\S3.3]{Mal12} are satisfied for any prederivator giving rise to a basic localiser (which requires additional hypotheses).

\begin{remark}\label{rk:adjointnotation}
If $\D$ is a prederivator and $u\colon J\to K$ a functor such that $u^\ast$ admits a left (resp.\,right) adjoint satisfying Der4L (resp.\,Der4R), we will still write $u_!$ (resp.\,$u_\ast$) for that adjoint. In particular, in the proposition above, we would write $\ell_!=r^\ast$ and~$r_\ast=\ell^\ast$. Left or right adjoints of $u^\ast$ arising from categorical adjoints to $u$ will always satisfy the appropriate version of Der4.
\end{remark}

\begin{cor}[Lemma~1.19(1),~\cite{Gro13}]\label{cor:initiallimit}
Suppose that $B$ is a category with a final object $b_1\in B$. Then for any $X\in\D(B)$, there is a natural isomorphism $b_1^\ast X\overset{\cong}\rightarrow\pi_{B,!}X$. Similarly, if $B$ is a category with an initial object $b_0\in B$, then there is a natural isomorphism $\pi_{B,\ast}X\overset{\cong}\rightarrow b_0^\ast X$.
\end{cor} 
The one-line proof is that the functor $b_1\colon e\to B$ is a right adjoint and $b_0\colon e\to B$ is a left adjoint. Applying Proposition~\ref{prop:adjointhomotopyexact} gives the required natural isomorphisms.

There is a particular derivator technique in the calculus of mates that we have not used yet, but will need to use throughout this section. We state it below as a general lemma.
\begin{lemma}\label{lemma:horzhomotopyexactcancel}
Let $\D$ be a \emph{left} derivator and consider the following square in $\Dia$:
\begin{equation*}
\vcenter{\xymatrix{
A\ar[r]^-{v}\ar[d]_-{p}&B\ar[d]^-{q}\ar@{}[dl]|\swtrans\ar@{}[dl]<-1.25ex>|\alpha\\
C\ar[r]_-{w}&D
}}
\end{equation*}
We may paste onto this square the comma category associated to any $c\in C$:
\begin{equation}\label{dia:horzpastingcancel}
\vcenter{\xymatrix{
(p/c)\ar[r]^-\pr\ar[d]_-{\pi_{(p/c)}}&\ar@{}[dl]|\swtrans\ar@{}[dl]<-1.25ex>|\gamma A\ar[r]^-{v}\ar[d]_-{p}&B\ar[d]^-{q}\ar@{}[dl]|\swtrans\ar@{}[dl]<-1.25ex>|\alpha\\
e\ar[r]_-c&C\ar[r]_-{w}&D
}}
\end{equation}
Then our original square is $\D$-exact if and only if the pasting
\begin{equation*}
\vcenter{\xymatrix{
(p/c)\ar[r]^-{v\text{\,pr}}\ar[d]_-{\pi_{(p/c)}}&B\ar[d]^-{q}\ar@{}[dl]|\swtrans\\
e\ar[r]_-{w(c)}&D
}}
\end{equation*}
(with natural transformation $\alpha\odot\gamma$) is $\D$-exact for every $c\in C$.
\end{lemma}
\begin{proof}
There is a small technical point that we will see proven shortly: though the pasting in $\Dia$ is $\alpha\odot\gamma$, after applying $\D$ we obtain a pasting $\gamma^\ast\odot\alpha^\ast$ because of the contravariance with respect to functors and covariance with respect to natural transformations. In our notation, we have $(\alpha\odot\gamma)_!=\gamma_!\odot\alpha_!$ by Proposition~\ref{prop:calculusofmates}(1). For both directions of the proof, by Der4L we know that $\gamma_!$ is a natural isomorphism.

For the forward direction, assume our original square was $\D$-exact, so that $\alpha_!$ is a natural isomorphism. Then by the reasoning above, the natural trans-\linebreak formation $(\alpha\odot\gamma)_!$ is the pasting of two natural isomorphisms, thus is itself a natural isomorphism which proves that the pasting is $\D$-exact.

We now turn to the converse, and assume that the pasting is $\D$-exact for every $c\in C$. We will look at the natural transformation $(\gamma\odot\alpha)_!=\alpha_!\odot\gamma_!$ applied to some $X\in\D(B)$ one step at a time. Applying the derivator $\D$ to Diagram~\ref{dia:horzpastingcancel} we obtain (after rotating)
\begin{equation*}
\vcenter{\xymatrix{
\D(D)\ar[r]^-{w^\ast}\ar[d]_-{q^\ast}&\ar@{}[dl]|\netrans\ar@{}[dl]<-1.25ex>|{\alpha^\ast}\D(C)\ar[r]^-{c^\ast}\ar[d]_-{p^\ast}&\ar@{}[dl]|\netrans\ar@{}[dl]<-1.25ex>|{\gamma^\ast}\D(e)\ar[d]^-{\pi_{(p/c)}^\ast}\\
\D(B)\ar[r]_-{v^\ast}&\D(A)\ar[r]_-{\pr^\ast}&\D((p/c))
}}
\end{equation*}
Taking left mates, we obtain
\begin{equation*}
\vcenter{\xymatrix{
\D(B)\ar[r]^-{q_!}\ar@(d,l)[dr]_-=&\ar@{}[dl]|\netrans\D(D)\ar[r]^-{w^\ast}\ar[d]_-{q^\ast}&\ar@{}[dl]|\netrans\ar@{}[dl]<-1.25ex>|{\alpha^\ast}\D(C)\ar@(r,u)[dr]^-=\ar[d]_-{p^\ast}&\ar@{}[dl]|\netrans\\
{}&\D(B)\ar[r]_-{v^\ast}&\D(A)\ar@(d,l)[dr]_-=\ar[r]_-{p_!}&\D(C)\ar[r]^-{c^\ast}\ar[d]_-{p^\ast}\ar@{}[dl]|\netrans&\ar@{}[dl]|\netrans\ar@{}[dl]<-1.25ex>|{\gamma^\ast}\D(e)\ar[d]_-{\pi_{(p/c)}^\ast}\ar@(r,u)[dr]^-=&{}\ar@{}[dl]|\netrans\\
&&&\D(A)\ar[r]_-{\pr^\ast}&\D((p/c))\ar[r]_-{\pi_{(p/c),!}}&\D(e)
}}
\end{equation*}

The top horizontal pasting is $\alpha_!$ and the bottom is $\gamma_!$. Moreover, the transformation in the middle $p^\ast\Rightarrow p^\ast p_! p^\ast\Rightarrow p^\ast$ is just the identity by the triangle identities of the adjunction $(p_!,p^\ast)$, so we have not introduced anything aberrant in this process. This is the key ingredient in the proof of the compatibility of the calculus of mates with pasting.

Rewriting this diagram with $\alpha_!$ and $\gamma_!$ and rotating it to our preferred orientation we obtain
\begin{equation*}
\vcenter{\xymatrix{
\D(B)\ar[r]^-{v^\ast}\ar[d]_-{q_!}&\ar@{}[dl]|\swtrans\ar@{}[dl]<-1.25ex>|{\alpha_!}\D(A)\ar[r]^-{\pr^\ast}\ar[d]_-{p_!}&\ar@{}[dl]|\swtrans\ar@{}[dl]<-1.25ex>|{\gamma_!}\D((p/c))\ar[d]^-{\pi_{(p/c),!}}\\
\D(D)\ar[r]_-{w^\ast}&\D(C)\ar[r]_-{c^\ast}&\D(e)
}}
\end{equation*}
Starting with any object $X\in\D(B)$, we obtain the following commutative diagram:
\begin{equation*}
\vcenter{\xymatrix@C=3em{
\pi_{(p/c),!}\pr^\ast v^\ast X\ar[d]_-=\ar[r]_-\cong^-{\gamma_{!,v^\ast X}}&c^\ast p_! v^\ast X\ar[r]^-{c^\ast\alpha_{!,X}}&c^\ast w^\ast q_! X\ar[d]^-=\\
\pi_{(p/c),!}(v\text{\,pr})^\ast X\ar[rr]_-\cong^-{(\gamma\odot\alpha)_{!,X}}&&w(c)^\ast q_! X
}}
\end{equation*}
We have isomorphisms where indicated because we know $\D$-exactness in these situations, and the vertical equalities follow from the strict 2-functoriality of any (pre)derivator. Thus $c^\ast\alpha_{!,X}$ must be an isomorphism for any $c\in C$ and $X\in\D(B)$.

In some other context, we might be stuck here, but we have at our disposal the axiom Der2. Rephrasing the axiom slightly for our purposes here, it states that a map~$f\colon X_1\to X_2$ in $\D(C)$ is an isomorphism if and only if $c^\ast f\colon c^\ast X_1\to c^\ast X_2$ is an isomorphism in $\D(e)$ for all $c\in C$. We apply this axiom to $f=\alpha_{!,X}$, $X_1=p_!v^\ast X$,\linebreak and $X_2=w^\ast q_! X$ to conclude that $\alpha_{!,X}$ is actually an isomorphism in $\D(C)$. Therefore $\alpha_!$ itself is a natural isomorphism without whiskering with $c^\ast$ and thus the corresponding square is $\D$-exact.
\end{proof}

\begin{lemma}\label{lemma:verthomotopyexactcancel}
In the situation of Lemma~\ref{lemma:horzhomotopyexactcancel}, we might have pasted vertically instead of horizontally to obtain for any $b\in B$:
\begin{equation*}
\vcenter{\xymatrix{
(b/v)\ar[r]^-{\pi_{(b/v)}}\ar[d]_-\pr&e\ar[d]^-b\ar@{}[dl]|\swtrans\ar@{}[dl]<-1.25ex>|\beta\\
A\ar[r]^-{v}\ar[d]_-{p}&B\ar[d]^-{q}\ar@{}[dl]|\swtrans\ar@{}[dl]<-1.25ex>|\alpha\\
C\ar[r]_-{w}&D
}}
\end{equation*}
Then if $\D$ is a \emph{right} derivator, the same conclusion holds: the total pasting is $\D$-exact for every $b\in B$ if and only if our original square is $\D$-exact.
\end{lemma}
\begin{proof}
The proof here involves instead the right mates of $\alpha$ and $\beta$ and that the top square is homotopy exact by Der4R. As such, we rotate our vertical pasting to become a horizontal one:
\begin{equation*}
\vcenter{\xymatrix{
(b/v)\ar[r]^-\pr\ar[d]_-{\pi_{(b/v)}}&A\ar[d]_-{v}\ar[r]^-p\ar@{}[dl]|\netrans\ar@{}[dl]<-1.25ex>|\beta&C\ar[d]^-w\ar@{}[dl]|\netrans\ar@{}[dl]<-1.25ex>|\alpha\\
e\ar[r]_-b&B\ar[r]_-q\ar[r]&D
}}
\end{equation*}

Following a sort of reasoning dual to the proof of Lemma~\ref{lemma:horzhomotopyexactcancel}, one obtains a commutative diagram for any $X\in \D(C)$,
\begin{equation*}
\vcenter{\xymatrix@C=3em{
\ar[d]_-=b^\ast q^\ast w_\ast X\ar[r]^-{b^\ast \alpha_{\ast,X}}&b^\ast v_\ast p^\ast X\ar[r]_-\cong^-{\beta_{\ast,p^\ast X}}&\pi_{(b/v),\ast}\pr^\ast p^\ast X\ar[d]^-=\\
q(b)^\ast w_\ast X\ar[rr]_-\cong^-{(\beta\odot\alpha)_{\ast,X}}&&\pi_{(b/v),\ast}(p\text{\,pr})^\ast X
}}
\end{equation*}
This proves that $\alpha_{\ast,X}$ is a pointwise isomorphism, thus by Der2, $\alpha_\ast$ is a natural isomorphism.
\end{proof}

If $\D$ is a (left and right) derivator, Lemma~\ref{lemma:horzhomotopyexactcancel} and  Lemma~\ref{lemma:verthomotopyexactcancel} both hold. However, we should not expect Lemma~\ref{lemma:verthomotopyexactcancel} to hold in an arbitrary left derivator, as the requisite right Kan extensions to form the right mate have no guarantee to exist. See \cite[Remarque~4.23]{Mal12} for a similar observation.

Nonetheless, the following theorem tells us that there is less of a distinction between Der4L and Der4R than this caution implies.

\begin{theorem}[Compare~Proposition~1.26,~\cite{Gro13} and~Th\'eor\`eme~4.24,~\cite{Mal12}]\label{thm:commaexact}
Let $\D$ be a semiderivator satisfying Der3L. Then $\D$ satisfies Der4L if and only if for any arbitrary comma square
\begin{equation*}\label{dia:laxpullback}
\vcenter{\xymatrix{
(u_1/u_2)\ar[r]^-{\pr_1}\ar[d]_-{\pr_2}&J_1\ar[d]^-{u_1}\ar@{}[dl]|\swtrans\ar@{}[dl]<-1.25ex>|\alpha\\
J_2\ar[r]_-{u_2}&K
}}
\end{equation*}
the left mate $\alpha_!$ is an isomorphism.
\end{theorem}
\begin{proof}
The backwards direction of this proof is automatic, as Der4L concerns a specific instance of the general oriented pullback square. Thus we assume that $\D$ satisfies Der4L, \ie $\D$ is a left derivator. We wlll begin by using Lemma~\ref{lemma:horzhomotopyexactcancel} and paste horizontally for some $j_2\in J_2$:
\begin{equation*}
\vcenter{\xymatrix{
(\pr_2/j_2)\ar[r]\ar[d]&\ar@{}[dl]|\swtrans\ar@{}[dl]<-1.25ex>|{\text{ex}}(u_1/u_2)\ar[r]^-{\pr_1}\ar[d]_-{\pr_2}&J_1\ar[d]^-{u_1}\ar@{}[dl]|\swtrans\ar@{}[dl]<-1.25ex>|\alpha\\
e\ar[r]_-{j_2}&J_2\ar[r]_-{u_2}&K
}}
\end{equation*}
We do not know that the outside square is $\D$-exact yet, so we will paste again on the left the comma square that would make the total pasting $\D$-exact by Der4L:
\begin{equation}\label{dia:doubleprojectionpaste}
\vcenter{\xymatrix{
(u_1/u_2(j_2))\ar[r]^-r\ar[d]&\ar@{}[dl]|\swtrans(\pr_2/j_2)\ar[r]\ar[d]&\ar@{}[dl]|\swtrans\ar@{}[dl]<-1.25ex>|{\text{ex}}(u_1/u_2)\ar[r]^-{\pr_1}\ar[d]_-{\pr_2}&J_1\ar[d]^-{u_1}\ar@{}[dl]|\swtrans\ar@{}[dl]<-1.25ex>|\alpha\\
e\ar[r]&e\ar[r]_-{j_2}&J_2\ar[r]_-{u_2}&K
}}
\end{equation}
Now the total pasting is $\D$-exact by Der4L, up to describing the map $r$. The codomain of $r$ consists of an object $(j'_1,j'_2,f\colon u_1(j'_1)\to u_2(j'_2))$ of $(u_1/u_2)$ along with a map $g\colon \pr_2(j'_1,j'_2,f)=j'_2\to j_2$. The domain of $r$ consists of an object $j''_1\in J_1$ and a map $h\colon u_1(j''_1)\to u_2(j_2)$ (where $j_2$ is fixed). We therefore define $r$ by
\begin{equation*}
r(j''_1,h\colon u_1(j''_1)\to u_2(j_2))=(j''_1,j_2,h\colon u_1(j''_1)\to u_2(j_2),j_2=j_2)
\end{equation*}

We claim that this map is a right adjoint. The left adjoint is given by
\begin{equation*}
\ell\left(j'_1,j'_2,f\colon u_1(j'_1)\to u_2(j'_2),g\colon j'_2\to j_2\right)=(j'_1,u_2(g)\circ f\colon u_1(j'_1)\to u_2(j'_2)\to u_2(j_2))
\end{equation*}
To sketch the bijection on hom-sets, let us describe the set\linebreak$\Hom_{(u_1/u_2(j_2))}(\ell(j'_1,j'_2,f,g),(j''_1,h))$. The maps are maps $a\colon j'_1\to j''_1$ in $J_1$ such that the following diagram commutes:
\begin{equation*}
\vcenter{\xymatrix{
u_1(j'_1)\ar[d]_-f\ar[r]^-{u_1(a)}&u_1(j''_1)\ar[d]^-h\\
u_2(j'_2)\ar[r]_-{u_2(g)}&u_2(j_2)
}}
\end{equation*}

Meanwhile, maps in $\Hom_{(\pr_2/j_2)}((j'_1,j'_2,f,g),r(j''_1,h))$ are maps in $(u_1/u_2)$ of the form $b\colon (j'_1,j'_2,f)\to~(j''_1,j_2,h)$, which themselves are certain pairs $b_1\colon j'_1\to j''_1$ and $b_2\colon j'_2\to j_2$ in $J_1$ and $J_2$ respectively, all making the following diagrams commute:
\begin{equation*}
\vcenter{\xymatrix{
u_1(j'_1)\ar[d]_-{u_1(b_1)}\ar[r]^-f &u_2(j'_2)\ar[d]^-{u_2(b_2)}\\
u_1(j''_1)\ar[r]_-h &u_2(j_2)
}}\quad\text{and}\quad
\vcenter{\xymatrix@R=0.5em{
j'_2\ar[dd]_-{b_2}\ar[dr]^-g\\
&j_2\\
j_2\ar[ur]_-{=}
}}
\end{equation*}
Commutativity of the second diagram means that $b_2=g$ is forced, thus the data is\linebreak $b_1\colon j'_1\to~j''_1$ making the square commute. But this is exactly the same as the other hom-set, proving a bijection. That this bijection is natural in both arguments is easy to verify. 

Since $r$ is a right adjoint, by Proposition~\ref{prop:adjointhomotopyexact} we know that the lefthand square of Diagram~\ref{dia:doubleprojectionpaste} is $\D$-exact. Thus in that diagram, the total pasting is $\D$-exact, as are the middle and lefthand squares, which implies that the righthand square is $\D$-exact as well, completing the proof.
\end{proof}

The dual statement is also true: for any right derivator $\D$, the right mate $\alpha_\ast$ as in the preceding proposition will be an isomorphism, which will use in the proof Lemma~\ref{lemma:verthomotopyexactcancel} instead. While Lemmata~\ref{lemma:horzhomotopyexactcancel}~and~\ref{lemma:verthomotopyexactcancel} seemed to be left- and right-dependent (respectively), the comma square of Theorem~\ref{thm:commaexact} is $\D$-exact for any half derivator $\D$. We have the following conclusion, recalling the hesitation of Remark~\ref{rk:homotopyexact} to apply the term `homotopy exact'.
\begin{cor}\label{cor:commaalwaysexact}
The oriented pullback square
\begin{equation*}
\vcenter{\xymatrix{
(u_1/u_2)\ar[r]^-{\pr_1}\ar[d]_-{\pr_2}&J_1\ar[d]^-{u_1}\ar@{}[dl]|\swtrans\ar@{}[dl]<-1.25ex>|\alpha\\
J_2\ar[r]_-{u_2}&K
}}
\end{equation*}
is $\D$-exact for any \emph{half} derivator $\D$.
\end{cor}

We have two more classes of squares which are $\D$-exact for any half derivator $\D$.

\begin{prop}[Compare Proposition~1.24,~\cite{Gro13}]\label{prop:grothendieckexactsquare} Consider a strict pullback square in $\Dia$ of the form
\begin{equation*}
\vcenter{\xymatrix{
A\ar[r]^-{v}\ar[d]_-{p}&B\ar[d]^-{q}\ar@{}[dl]|\swtrans\ar@{}[dl]<-1.25ex>|\id\\
C\ar[r]_-{w}&D
}}
\end{equation*}
Then this square is $\D$-exact for every left derivator $\D$ whenever $q$ is a Grothendieck opfibration. It is $\D$-exact for every right derivator $\D$ whenever $w$ is a Grothendieck fibration.
\end{prop}

\begin{proof}
We prove the left case and indicate the necessary modifications for the right case. Our general strategy is to horizontally paste using Lemma~\ref{lemma:horzhomotopyexactcancel} and make sure we are adding $\D$-exact squares every time.

To begin, we note that that both Grothendieck fibrations and opfibrations are preserved under pullback, so we have that both $q$ and $p$ are opfibrations. We paste by the $\D$-exact square from Der4L associated to any $c\in C$:
\begin{equation*}
\vcenter{\xymatrix{
(p/c)\ar[r]^-\pr\ar[d]_-{\pi_{(p/c)}}&A\ar[r]^-{v}\ar[d]_-{p}\ar@{}[dl]|\swtrans\ar@{}[dl]<-1.25ex>|{\text{ex}}&B\ar[d]^-{q}\ar@{}[dl]|\swtrans\ar@{}[dl]<-1.25ex>|\id\\
e\ar[r]_-c&C\ar[r]_-{w}&D
}}
\end{equation*}
The outside square is not yet comprehensible, however. We therefore consider the strict fibre of $p$ over $c\in C$, namely, the non-full subcategory $p^{-1}(c)\subset A$ on $a\in A$ such that $p(a)=c$ and maps $f\colon a\to a'$ such that $p(f)=\id_c$. There is a map $s_c\colon p^{-1}(c)\to (p/c)$ such that $s(a)=(a,\id_c)$ which compares these two categories. If $p$ is a Grothendieck opfibration, then in particular $p\colon A\to C$ is a precofibered category (using the terminology from \cite{SGA1}. A precofibered category by definition is one such that these inclusion functors admit a left adjoint for every $c\in C$. Being able to lift the map $p(a)\to c=p(a')$ in $C$ to a cocartesian morphism in $A$ gives the adjoint.

By Proposition~\ref{prop:adjointhomotopyexact} the square pasted on the left below is also $\D$-exact:
\begin{equation}\label{dia:pasting1}
\vcenter{\xymatrix@C=3em{
p^{-1}(c)\ar[r]^-{s_c}\ar[d]_-{\pi_{p^{-1}(c)}}&(p/c)\ar[r]^-\pr\ar[d]_-{\pi_{(p/c)}}\ar@{}[dl]|\swtrans\ar@{}[dl]<-1.25ex>|{\text{ex}}&A\ar[r]^-{v}\ar[d]_-{p}\ar@{}[dl]|\swtrans\ar@{}[dl]<-1.25ex>|{\text{ex}}&B\ar[d]^-{q}\ar@{}[dl]|\swtrans\ar@{}[dl]<-1.25ex>|\id\\
e\ar[r]&e\ar[r]_-c&C\ar[r]_-{w}&D
}}
\end{equation}
We need only prove that the outside square is $\D$-exact, and we do so by creating a pasting equal to the above. Since our original square commuted on the nose, we know that $w(p(a))=q(v(a))$ for any $a\in A$. Therefore the functor $v\colon A\to B$ restricts to the strict fibres $v\colon p^{-1}(c)\to q^{-1}(w(c))$. We can also have a right adjoint functor $s_{w(c)}\colon q^{-1}(w(c))\to (q/w(c))$ because $q$ is a Grothendieck opfibration. In total, we get a pasting in which the middle and right squares are $\D$-exact.
\begin{equation}\label{dia:pasting2}
\vcenter{\xymatrix@C=5em{
p^{-1}(c)\ar[r]^-{v}\ar[d]_-{\pi_{p^{-1}(c)}}&q^{-1}(w(c))\ar[d]_-{\pi_{q^{-1}(w(c))}}\ar[r]^-{s_{w(c)}}\ar@{}[dl]|\swtrans\ar@{}[dl]<-1.25ex>|{\text{ex}}&(q/w(c))\ar[r]^-\pr\ar[d]_-{\pi_{(p/c)}}\ar@{}[dl]|\swtrans\ar@{}[dl]<-1.25ex>|{\text{ex}}&B\ar[d]^-{q}\ar@{}[dl]|\swtrans\ar@{}[dl]<-1.25ex>|{\text{ex}}\\
e\ar[r]&e\ar[r]&e\ar[r]_-{w(c)}&D
}}
\end{equation}
This is the same total pasting as Diagram~\ref{dia:pasting1}. To see this, let $a\in p^{-1}(c)$. The top composition of Diagram~\ref{dia:pasting1} yields $q(v(\pr(s_c(a))))=q(v(a))$, and the bottom composition yields $w(p(a))$. Because our original square commuted on the nose, these are equal. Moreover, the map $q(v(a))\to w(p(a))$ must be the identity: the only data comes from the structure map in $s_c(a)=(a,\id_c)$.

For Diagram~\ref{dia:pasting2}, the top composition gives $q(\pr(s_{w(c)}(v(a))))=q(v(a))$ and the bottom gives $w(c)=w(p(a))$ again. The data of the map $q(v(a))\to w(p(a))$ is the structure map in $s_{w(c)}(v(a))=(v(a),q(v(a))=w(c))$.

The final thing to check is that the lefthand square of Diagram~\ref{dia:pasting2} is exact. But this is exactly our original pullback square of categories restricted to $c\in C$ and its image $w(c)\in D$. This means that $v\colon p\inv(c)\to q\inv(w(c))$ is an isomorphism (since it is the pullback of $e\cong e$), so it is a right adjoint functor and by Proposition~\ref{prop:adjointhomotopyexact} we are done.\

In the case that $\D$ is a right derivator, we paste vertically using the Der4R square associated to $b\in B$ and use Lemma~\ref{lemma:verthomotopyexactcancel} instead. That $u,v$ are Grothendieck fibrations (and a fortiori prefibred categories) means that the inclusion of the strict fibres $s_b\colon v^{-1}(b)\to (v/b)$ and $s_{q(b)}\colon w\inv(q(b))\to (w/q(b))$ are \emph{left} adjoints, allowing for Proposition~\ref{prop:adjointhomotopyexact} again.
\end{proof}

\begin{remark}
As a first remark, a version of this proposition can be read from \cite[Lemme~3.12]{Mal12} when $\D$ gives rise to a basic localiser.

Second, we did not actually require the full strength of a Grothendieck (op)fibration, as pre(co)fibered categories are still preserved under pullbacks, see \cite[D\'efinition~6.1 and Corollaire~6.9]{SGA1}. Maltsiniotis also proves his lemma under this weaker condition.

In practice we will run into Grothendieck (op)fibrations only and will not need these weaker hypotheses, but it is good to keep our options open.
\end{remark}
We recover one more result in the half derivator case.

\begin{prop}[Compare Proposition~1.20,~\cite{Gro13}]\label{prop:fullyfaithfulkanextension}
Let $u\colon J\to K$ be a fully faithful functor. Then the following commutative square in $\Dia$ is $\D$-exact for any half derivator $\D$:
\begin{equation*}
\vcenter{\xymatrix{
J\ar[r]^-{\id_J}\ar[d]_-{\id_J}&J\ar[d]^-u\ar@{}[dl]|\swtrans\ar@{}[dl]<-1.25ex>|{\id}\\
J\ar[r]_-u&K
}}
\end{equation*}
Specifically, the left mate is the counit of the $(u^\ast,u_\ast)$ adjunction and right mate is the unit of the $(u^\ast,u_\ast)$ adjunction. Since these are natural isomorphisms, we have that $u_\ast,u_!\colon\D(J)\to\D(K)$ are both fully faithful (when they exist).
\end{prop}
\begin{proof}
We prove the case where $\D$ is a left derivator and will use a similar proof to that of Theorem~\ref{thm:commaexact}. We start by pasting the $\D$-exact comma square on the left associated to some $j\in J$:
\begin{equation*}
\vcenter{\xymatrix@C=4em{
(\id_J/j)\ar[r]^-{\pr}\ar[d]_-{\pi_{(\id_J/j)}}&J\ar[r]^-{\id_J}\ar[d]_-{\id_J}\ar@{}[dl]|\swtrans\ar@{}[dl]<-1.25ex>|{\text{ex}}&J\ar[d]^-u\ar@{}[dl]|\swtrans\ar@{}[dl]<-1.25ex>|{\id}\\
e\ar[r]_-j&J\ar[r]_-u&K
}}
\end{equation*}
We now want to think about the category $(\id_J/j)$. Its objects are $j'\in J$ along with a map $f\colon j'\to j$, and its maps are morphisms over $j$. The functor $u\colon J\to K$ induces a functor $(\id_J/j)\to (u/u(j))$, where $(j',f)$ gets mapped to $(j',u(f))$. This map is an isomorphism of categories because $u$ is fully faithful: any map $u(j')\to u(j)$ in $K$ must come from a map $j'\to j$ in $J$ and uniquely so. Let $v\colon (u/u(j))\to (\id_J/j)$ be the inverse of this functor. Pasting again, we have
\begin{equation}\label{dia:threesquares}
\vcenter{\xymatrix@C=4em{
(u/u(j))\ar[r]^-v_-\cong\ar[d]_-{\pi_{(u/u(j))}}&(\id_J/j)\ar[r]^-{\pr}\ar[d]_-{\pi_{(\id_J/j)}}\ar@{}[dl]|\swtrans\ar@{}[dl]<-1.25ex>|{\text{ex}}&J\ar[r]^-{\id_J}\ar[d]_-{\id_J}\ar@{}[dl]|\swtrans\ar@{}[dl]<-1.25ex>|{\text{ex}}&J\ar[d]^-u\ar@{}[dl]|\swtrans\ar@{}[dl]<-1.25ex>|{\id}\\
e\ar[r]_-{\id_e}&e\ar[r]_-j&J\ar[r]_-u&K
}}
\end{equation}
The square we have pasted is homotopy exact because $v$ is an isomorphism, so in particular a right adjoint functor and Proposition~\ref{prop:adjointhomotopyexact} applies. Using two applications of Lemma~\ref{lemma:horzhomotopyexactcancel}, our original square is homotopy exact if and only if the total pasting is homotopy exact.

What is this total horizontal pasting? It is
\begin{equation*}
\vcenter{\xymatrix{
(u/u(j))\ar[r]^-\pr\ar[d]_-\pi&J\ar[d]^-u\ar@{}[dl]|\swtrans\ar@{}[dl]<-1.25ex>|{\text{ex}}\\
e\ar[r]_-{u(j)}&K
}}
\end{equation*}
which we identify as the Der4L comma square associated to $u(j)\in K$. This square is therefore $\D$-exact, and we are done.
\end{proof}

\section{Pointed derivators}\label{sec:pointedders} 

The derivators we study in Section~\ref{sec:halfpointed} will satisfy an additional axiom.
\begin{defn}\label{defn:pointed}
A derivator $\D$ is \emph{pointed} if its underlying category $\D(e)$ is pointed. That is, the unique morphism from the intial object to the final object an isomorphism. We will write $0\in\D(e)$ for this zero object.
\end{defn}

\begin{ex}\label{ex:pointedderivators}\ \\\vspace{-1em}
\begin{enumerate}
\item For a bicomplete category $\cC$, the represented derivator $\D_\cC$ is a pointed derivator if $\cC$ is pointed as a category.
\item For a small Grothendieck abelian category $\cA$, the derivator $\D_\cA$ is a pointed derivator.
\item For a combinatorial model category $\cM$, the derivator $\D_\cM$ is pointed if $\cM$ is pointed.
\end{enumerate}
\end{ex}
Note that if $\D$ is a derivator, $\D(e)$ always admits an initial and a final object. Let~$\varnothing$ denote the empty category. By Der1, we have an equivalence of categories 
\begin{equation*}
\vcenter{\xymatrix{
\D(\varnothing)=\D(\varnothing\sqcup\varnothing)\ar[r]^-\sim&\D(\varnothing)\times\D(\varnothing)
}}
\end{equation*}
which implies that $\D(\varnothing)$ is equivalent to $e$. Let $\pi_\varnothing\colon\varnothing\to e$ be the unique functor in $\Dia$. Then the left and right Kan extension along $\pi_\varnothing$ have the form $e\to\D(e)$, so pick out a single object in $\D(e)$. These are the final and initial objects, the final object being the empty limit and the initial object the empty colimit. Whether these are isomorphic is again a property of the derivator $\D$ (specifically, a property of its underlying category).

Definition~\ref{defn:pointed} used to define the adjective \emph{weakly pointed}. There is an obvious way to strengthen this axiom: we ask that for all $K\in\Dia$, $\D(K)$ is a pointed category, and that $u^\ast\colon\D(K)\to\D(J)$ is a pointed functor for any $u\colon J\to K$. If~$\D$ is a full derivator, this is automatic. Let $\pi_K\colon K\to e$ be the projection to the point. Then the pullback~$\pi_K^\ast$ is both a left and right adjoint in any derivator $\D$, so $0_K:=\pi_K^\ast(0)$ should be both an initial and final object in $\D(K)$, meaning that $\D(K)$ is also pointed. Similarly, $u^\ast\colon \D(K)\to\D(J)$ is both a left and right adjoint, so it preserves initial and final objects, so it sends $0_K$ to $0_J$. In fact, $u_!,u_\ast\colon\D(J)\to\D(K)$ are also pointed because each is an adjoint functor.

\begin{remark}
If $\D$ is a half derivator such that $\D(e)$ is pointed, we cannot use the above argument to prove that each $\D(K)$ is pointed and that each $u^\ast$ is a pointed functor. Nonetheless this is true; we prove this in Theorem~\ref{thm:halfstronglypointed} below. We hesitate to call this a left/right pointed derivator, as we reserve this terminology for a more specific notion in Definition~\ref{defn:leftpointedder}. However, we can still prove statements about these objects, so we will continue to write the lengthy phrase `half derivator $\D$ such that $\D(e)$ is pointed' for the remainder of this section.
\end{remark}

There is also a notion of a \emph{strongly pointed} derivator, which was originally called\linebreak Der6 in \cite{Mal07}. In order to state it, we need to recall two particular classes of functors in $\Cat$.

\begin{defn}
Let $u\colon J\to K$ be a fully faithful functor that is injective on objects.
\begin{enumerate}
\item The functor $u$ is a \emph{sieve} if for any morphism $k\to u(j)$ in $K$, $k$ lies in the image of $u$.
\item The functor $u$ is a \emph{cosieve} if for any morphism $u(j)\to k$ in $K$, $k$ lies in the image of $u$.
\end{enumerate}
\end{defn}
\begin{defn}\label{defn:extraadjoint}
\cite[p.6]{Mal07} A derivator $\D$ is \emph{strongly pointed} if for every sieve (resp. cosieve) $u\colon J\to K$ in $\Dia$, $u_\ast$ (resp. $u_!$) admits a right adjoint $u^!$ (resp. left adjoint~$u^?$).
\end{defn}

Asking for these exceptional adjoints is confusing until we know more about the right Kan extension along a sieve or the left Kan extension along a cosieve.

\begin{prop}[Compare Proposition~1.23,~\cite{Gro13}]\label{prop:extensionbyzero}
Let $\D$ be a right (resp.\,left) derivator such that $\D(e)$ is pointed, and let $u\colon J\to K$ be a sieve (resp. cosieve). Then $u_\ast\colon\D(J)\to\D(K)$ (resp. $u_!$) is fully faithful, with essential image $X\in \D(K)$ such that $k^\ast X\cong 0$ for all $k\in K\setminus u(J)$.
\end{prop}
\begin{proof}
We will prove the first case, since these are exactly the right Kan extensions we will need for derivator K-theory. Let $\D$ be a right derivator such that $\D(e)$ is pointed, $u\colon J\to K$ a sieve, and $X\in\D(J)$. Then we can examine $u_\ast X$ pointwise using Der4R, which involves understanding the comma category $(k/u)$ for all $k\in K$. Recall that its objects are pairs $j\in J$ with a morphism $k\to u(j)$ and its maps are maps in $J$ under $k$.

Suppose that $k$ is not in the image of $u$. Then because $u$ is a sieve, there cannot be any maps $k\to u(j)$ for any $j\in J$, so the comma category $(k/u)$ is empty. Hence
\begin{equation*}
k^\ast u_\ast X\cong \pi_{\varnothing,\ast}\pr^\ast X\cong 0
\end{equation*}
because the right Kan extension along $\pi_\varnothing\colon\varnothing\to e$ gives the final object, \ie the zero object when $\D(e)$ is pointed.

Now if $k$ is in the image of $u$, write $k=u(j')$ and consider any\linebreak object ${(j,f\colon k\to u(j))}$ in the comma category $(k/u)$. Because $u$ is fully faithful, any map $f\colon k=u(j')\to u(j)$ in $K$ must be the image of a map $f'\colon j'\to j$ in $J$. Thus we can consider $(k/u)$ to have objects $(j,f'\colon j'\to j)$ and maps in $J$ under $j'$. This category admits the initial object $(j',\id_{j'})$. For any $(j,f'\colon j'\to j)$ we have the unique map from $(j',\id_{j'})$ given by
\begin{equation*}
\vcenter{\xymatrix@C=1em{
&j'\ar[dr]^-{f'}\ar[dl]_-=\\
j'\ar[rr]_-{f'}&&j
}}
\end{equation*}
We can now take advantage of a categorical adjunction: the inclusion of the initial object $(j',\id_{j'})$ is left adjoint to the projection $\pi_{(k/u)}\colon (k/u)\to e$. This means that, upon applying $\D$, we obtain an adjunction $(\pi_{(k/u)}^\ast,(j',\id_{j'})^\ast)$. Because adjoints are unique up to unique isomorphism, we conclude that the limit $\pi_{(k/u),\ast}$ is canonically isomorphic to $(j',\id_{j'})^\ast$. This is the `long proof' of Corollary~\ref{cor:initiallimit} in action.

We obtain the following chain of isomorphisms:
\begin{equation*}
\vcenter{\xymatrix{
k^\ast u_\ast X\ar[r]^-{\cong}& \pi_{(k/u),\ast}\pr^\ast X\ar[r]^-\cong& \pi_{e,\ast}(j',\id_{j'})^\ast \pr^\ast X
}}
\end{equation*}
We have no clear sense of what $\pr^\ast X$ looks like, but we know the compo-\linebreak site~$\pr\circ(j',\id_{j'})\colon e\to~J$ is the inclusion of the object $j'$. Hence $(j',\id_{j'})^\ast \pr^\ast X=j'^\ast X$. Finally, since $\pi_e\colon e\to e$ is just a fancier way of writing $\id_e$, we have
\begin{equation*}
k^\ast u_\ast X\cong \pi_{e,\ast}(j',\id_{j'})^\ast \pr^\ast X= \id_{e,\ast} j'^\ast X\cong j'^\ast X
\end{equation*}
Thus when $k\in K$ is in the image of $u$, the value of $u_\ast X$ at $k$ is exactly what it was in $\D(J)$ under the fully faithful inclusion $u\colon J\to K$.

Combining the above with Proposition~\ref{prop:fullyfaithfulkanextension}, we complete the proof. The case for left derivators and left Kan extensions along cosieves follows by a dual argument.
\end{proof}

We call these \emph{extension by zero functors}, and the corresponding morphisms of derivators $\D^J\to\D^K$ \emph{extension by zero morphisms}. We have the following surprising result:
\begin{prop}[Corollaries~3.5~and~3.8, \cite{Gro13}]\label{prop:stronglypointed}
A (left and right) derivator is pointed if and only if it is strongly pointed.
\end{prop}
The backwards direction is easy: using the exceptional adjoints to the (co)sieve $\varnothing\to e$, we can show that the initial object is also final. The forwards direction is difficult, but we will not need the details for our main theorem.

\section{Half pointed derivators}\label{sec:halfpointed}

To motivate the following definition, we recall the definition of $K_0$ of an abelian category~$\cA$. It is constructed as the free abelian group on (isomorphism classes of) objects $A\in \cA$, written $[A]\in K_0(\cA)$, under the relation that if $0\to A\to B\to C\to 0$ is a short exact sequence, we have $[B]=[A]+[C]$. A short exact sequence is equivalently a cocartesian square
\begin{equation*}
\vcenter{\xymatrix@C=1em@R=1em{
A\ar[r]\ar[d]&B\ar[d]\\
0\ar[r]&C
}}
\end{equation*}
under the assumption that $A\to B$ is a monomorphism. Thus if we are to construct even $K_0$ for a derivator, it needs to admit a notion of (coherent) cocartesian squares and a zero object.

\begin{notn}\label{notn:square}
Let $\square$ be the category
\begin{equation*}
\xymatrix@C=1em@R=1em{
(0,0)\ar[r]\ar[d]&(1,0)\ar[d]\\
(0,1)\ar[r]&(1,1)
}
\end{equation*}
Let $i_{\ul}\colon\ul\to\square$ be the full subcategory lacking the element $(1,1)$.
\end{notn}

\begin{defn}
Let $\D$ be a pointed derivator and $X\in\D(\square)$. We say that $X$ is \emph{cocartesian} (\ie a pushout square) if $X$ is in the essential image of $i_{\ul,!}\colon\D(\ul)\to\D(\square)$.
\end{defn}

In particular, we can construct pushouts appropriate for computing $K_0$ as above by constructing cocartesian squares starting from an element in $\D(\ul)$ of the form
\begin{equation*}
\vcenter{\xymatrix@R=1em@C=1em{
A\ar[r]\ar[d]&B\\
0
}}
\end{equation*}
We will see soon how to construct such objects coherently.

\begin{defn}\label{defn:leftpointedder}
A prederivator $\D\colon\Dia\op\to\CAT$ is a \emph{left pointed derivator} if it is a left derivator, $\D(e)$ is pointed, and for every sieve $u\colon J\to K$, $u^\ast$ admits a right adjoint $u_\ast$.

A prederivator $\D\colon\Dia\op\to\CAT$ is a \emph{right pointed derivator} if it is a right derivator, $\D(e)$ is pointed, and for every cosieve $u\colon J\to K$, $u^\ast$ admits a left adjoint $u_!$.

If we are referring to either a left or a right pointed derivator we will use the general term \emph{half pointed derivator}.
\end{defn}

\begin{remark}
We proved in Corollary~\ref{cor:commaalwaysexact} that the right Kan extensions in a left pointed derivator must satisfy Der4R and the left Kan extensions in a right pointed derivator must satisfy Der4L.
\end{remark}

Left pointed derivators are the important class for derivator K-theory, so we will focus more on them. The existence of the specified $u_\ast$ means that a left pointed derivator $\D$ admits right extension by zero morphisms (recall Proposition~\ref{prop:extensionbyzero}) that can be computed pointwise.

Here is why we need the extension by zero morphisms: suppose we have a coherent map $(f\colon a\to b)\in\D^{[1]}$ and want to compute its cofibre, \ie construct the cocartesian square
\begin{equation}\label{dia:cofibersquare}
\vcenter{\xymatrix@C=1em@R=1em{
a\ar[r]^f\ar[d]&b\ar[d]\\
0\ar[r]&C(f)
}}
\end{equation}
We first have to extend by zero using the functor $i_{[1]}\colon [1]\to\ul$ which includes into the horizontal arrow. But this is a sieve, so the extension by zero morphism is the \emph{right} Kan extension. Thus in an ordinary left derivator, we would not have access to the functor $i_{[1],\ast}\colon \D^{[1]}\to\D^\ul$. After this we may compute the pushout using $i_{\ul,!}\colon\D^{\ul}\to\D^\square$. These two steps yield the above (coherent) cocartesian square
which should be an important part of defining $K_0(\D)$. Without access to the extension by zero morphism $i_{[1],\ast}$ we would be lost from the outset.

But all this discussion begs the question: why not just use an ordinary pointed derivator, which would have all the right Kan extensions we could ever want? There is an important class of examples that are \emph{not} full derivators but are left pointed, which we hinted at in Remark~\ref{rk:dirf}.
\begin{lemma}\label{lemma:waldhausenderivator}
Let $\cW$ be a saturated Waldhausen category satisfying the cylinder axiom. Then the associated prederivator
\begin{equation*}
\D_\cW\colon K\mapsto \Ho(\Fun(K,\cW))
\end{equation*}
defined on $\Dirf$ is a (strong in the sense of Cisinski) left pointed derivator.
\end{lemma}
\begin{proof}
This follows by combining Exemple~2.23, Corollaire~2.24, and Lemme~4.3 of \cite{Cis10}. Note that for Cisinski, anything \emph{\`a droite} is what we would call \emph{left}, \eg \cite[D\'efinition~A.6]{Cis10} defining a \emph{d\'erivateur faible \`a droite} matches our Definition~\ref{defn:leftderivator} defining a left derivator.
\end{proof}

Recall that $\Dirf$ is the diagram 2-category of finite direct categories. These derivators $\D_\cW$ cannot be full derivators in general because arbitrary Waldhausen categories admit no notion of product or pullback square, although Waldhausen categories whose structure descends from exact or abelian categories do. To challenge our intuition, recall that in \cite[\S1.7]{Wal85} Waldhausen defines for any Waldhausen category $\cW$ and (co)homology theory on $\cW$ the subcategory of $n$-spherical objects, \ie objects which (co)homology concentrated in degree $n$. These categories are still Waldhausen but do not admit products, even when $\cW$ is an abelian category, as the (co)homology of a product need not remain in degree $n$. Waldhausen uses these cellular filtrations to prove various theorems in the rest of Chapter~1 of \cite{Wal85}. As an analogous example, categories of chain complexes valued in an abelian or exact category with bounded (co)homology are Waldhausen categories that do not admit all (fibre) products. 

Investigating derivator K-theory as an extension of Waldhausen K-theory is only possible if we are as general as possible with the input to derivator K-theory.

\begin{remark}
A note about historical definitions is in order here. Heller in~\cite{Hel88} defines a `left homotopy theory' to be (using modern terminology) a left derivator such that every discrete fibration $p\colon E\to B$ admits a right Kan extension along with a strange condition on the underlying diagram functors $\dia_E,\dia_B$ (his axiom H4L). In particular, this implies that left homotopy theories admit finite products, a situation we do not wish to replicate. But because sieves are a particular kind of discrete fibration, one in which the fibres are either a singleton or empty, one can view our definition as a weakening of Heller's (with the addition of `pointed').

On the other hand, in the reformulation of derivator K-theory in~\cite{MurRap17}, the authors define derivator K-theory on all `pointed right derivators', which in our terms (as we warned in Remark~\ref{rk:halfderivator}) is a left derivator $\D$ such that $\D(e)$ is pointed. For us, this is not enough structure. The cocartesian squares of the form of Diagram~\ref{dia:cofibersquare} still exist in such a derivator, as they are identifiable as the essential image of the pushout functor $i_{\ul,!}\colon \D(\ul)\to\D(\square)$ satisfying $(0,1)^\ast X=0\in\D(e)$. However, they are impossible to construct given just the information of $(f\colon a\to b)\in\D^{[1]}$. 

We choose the order of our adjectives to emphasize that $\D$ is half of a pointed derivator, not a half derivator that happens to be pointed (as in~\cite{MurRap17} and Proposition~\ref{prop:extensionbyzero}). In particular, a Muro-Raptis `pointed right derivator' may not be strongly pointed, as the construction of exceptional left adjoints to left extension by zero morphisms requires particular right extension by zero morphisms to exist. However, our left pointed derivators do not have this (potential) defect, as the following theorem shows.
\end{remark}

\begin{theorem}\label{thm:halfstronglypointed}
A half pointed derivator $\D$ is half strongly pointed. That is, for a half pointed derivator $\D$, each category $\D(K)$ is pointed and $u^\ast$ is a pointed functor for all functors $u$.

Moreover, in a left pointed derivator, every $u_!$ is pointed functor and left Kan extensions along cosieves admit exceptional left adjoints. In a right pointed derivator, every $u_\ast$ is a pointed functor and right Kan extensions along sieves admit exceptional right adjoints.
\end{theorem}
\begin{proof}
We begin by showing that each $\D(K)$ of a half pointed derivator is pointed. The map $i_\varnothing\colon\varnothing\to K$ is both a sieve and a cosieve, so the\linebreak functors $i_{\varnothing,\ast},i_{\varnothing,!}\colon\D(\varnothing)\to \D(K)$ defining the initial and final object in $\D(K)$ always exist for any half pointed derivator. Moreover, they are isomorphic in any prederivator satisfying Der2 with pointed base: there is a canonical transformation $i_{\varnothing,\ast}\Rightarrow i_{\varnothing,!}$ which is an isomorphism when whiskered with $k^\ast$ for every $k\in K$. We have an isomorphism $k^\ast i_{\varnothing,!}\cong \pi_{\varnothing,!}$ where $\pi_\varnothing\colon\varnothing\to e$ is the empty inclusion/projection into the final category, and similarly $k^\ast i_{\varnothing,\ast}\cong \pi_{\varnothing,\ast}$. Because $\D(e)$ is pointed, the canonical natural transformation $\pi_{\varnothing,!}\Rightarrow \pi_{\varnothing,\ast}$ is an isomorphism, hence the same thing is true for $k^\ast i_{\varnothing,\ast}\Rightarrow k^\ast i_{\varnothing,!}$

By Der2, this means that initial objects in $\D(K)$ are also final, making $\D(K)$ a pointed category. Left and right Kan extensions preserve final and initial objects (respectively), so they also preserve zero objects. Pullback functors preserve either initial or final objects depending on whether $\D$ is a left or right pointed derivator, but in either case they preserve zero objects as well.

Finally, the construction of the exceptional left adjoint to a cosieve requires only left Kan extensions and right Kan extensions along sieves. Therefore the same construction still works in a left pointed derivator. For more details in both the left and right cases see~\cite[\S3]{Gro13} and Corollary~3.8 in particular.
\end{proof}

\begin{remark}
We might hope that in a left pointed derivator, right Kan extensions along sieves would still admit exceptional right adjoints. There does not seem any reason for this to be true. To give an explicit counterexample, consider the sieve $i_{[1]}\colon[1]\to\ul$ and consider an object $\D(\ul)$ given (incoherently) by
\begin{equation*}
\vcenter{\xymatrix@R=1em@C=1em{
&a\ar[r]\ar[dl]&b\\
c
}}
\end{equation*}
Then the formula for $i_{[1]}^!$ is the following composition:
\begin{equation*}
\vcenter{\xymatrix@R=0.5em@C=1em{
&a\ar[r]\ar[dl]&b\\
c
}}\mapsto
\vcenter{\xymatrix@R=0.5em@C=1em{
0\ar[dd]\\
&a\ar[r]\ar[dl]&b\\
c
}}\mapsto
\vcenter{\xymatrix@R=0.5em@C=1em{
&P\ar[r]\ar[dd]\ar[dl]&b\ar[dd]^=\\
0\ar[dd]\\
&a\ar[r]\ar[dl]&b\\
c
}}\mapsto
\vcenter{\xymatrix@R=0.5em@C=1em{
P\ar[r]&b
}}
\end{equation*}
The lefthand (slanted) square in the third step is cartesian, \ie $P$ is the pullback of $a\to c$ along zero. But in an arbitrary left pointed derivator, this pullback has no reason to exist. The details of this construction can be found in \cite[pp.62-64]{Col19b}.
\end{remark}

Having established the objects we will use as input to derivator K-theory, we now need to discuss the morphisms. Such morphisms will need to preserve cocartesian squares and the zero object, so that they preserve objects like Diagram~\ref{dia:cofibersquare}. The common adjective for such morphisms is \emph{right exact}, which will remain (for the moment) an example of the left/right notational hazard.

Any cocontinuous morphism will be right exact, \ie pullbacks of functors $u\colon J\to~K$ in $\Dia$ and the associated left Kan extensions. But we no longer have exceptional right adjoints to our right Kan extensions along sieves, so these morphisms are not left adjoints and hence not automatically cocontinuous. Fortunately, we have the following.

\begin{theorem}\label{thm:extbyzerococ}
Let $\D$ be a left pointed derivator. Then for any sieve $i\colon A\to B$, $i_\ast\colon \D^A\to\D^B$ is a cocontinuous morphism of derivators.
\end{theorem}
\begin{proof}
It is enough to show that for any $K\in\Dia$, $i_\ast$ preserves colimits of shape $K$ by Proposition~\ref{prop:colimitpreserving}. Specifically, if we let $\pi_K\colon K\to e$ be the projection, we need the canonical comparison map
\begin{equation*}
(\id_B\times\pi_K)_!(i\times\id_K)_\ast X\to i_\ast(\id_A\times\pi_K)_!X
\end{equation*}
to be an isomorphism for any $X\in\D(A\times K)$. It suffices to check that this map is an isomorphism at every $b\in B$ by Der2, so therefore we consider
\begin{equation}\label{eq:cantrans}
b^\ast(\id_B\times\pi_K)_!(i\times\id_K)_\ast X\to b^\ast i_\ast(\id_A\times\pi_K)_! X
\end{equation}
for any $X\in\D(A\times K)$.

We begin by analyzing the lefthand side. Even though $b^\ast$ does not necessarily admit a right adjoint, it is still cocontinuous by~\cite[Proposition~2.5]{Gro13}. That proposition is stated for full derivators, but its proof only depends on Theorem~\ref{thm:commaexact}, which we have shown is satisfied for half derivators.  Therefore
\begin{equation*}
b^\ast(\id_B\times\pi_K)_!(i\times\id_J)_\ast X\cong \pi_{K,!} (b\times \id_J)^\ast (i\times \id_J)_\ast X.
\end{equation*}
Using the properties of the extension by zero morphism $i_\ast$ from Proposition~\ref{prop:extensionbyzero}, we have that
\begin{equation*}
\pi_{K,!}(b\times \id_K)^\ast (i\times \id_K)_\ast X\cong\begin{cases} \pi_{K,!}(b\times\id_K)^\ast X&b\in A\\0&b\notin A\end{cases}
\end{equation*}

For the righthand side of Equation~\ref{eq:cantrans}, again using that $i_\ast$ is extension by zero,
\begin{equation*}
b^\ast i_\ast(\id_A\times\pi_K)_! X\cong\begin{cases}b^\ast (\id_A\times\pi_K)_! X&b\in A\\0&b\notin A\end{cases}
\end{equation*}

So if $b\notin A$, the canonical transformation in Equation~\ref{eq:cantrans} is an isomorphism of zero objects. If $b\in A$, then the transformation is isomorphic to
\begin{equation*}
\pi_{K,!}(b\times \id_K)^\ast X\to~b^\ast(\id_A\times \pi_K)_! X
\end{equation*}
which is still an isomorphism because the square
\begin{equation*}
\vcenter{\xymatrix@C=3em{
K\ar[r]^-{b\times\id_K}\ar[d]_-{\pi_K}&B\times K\ar[d]^-{\id_B\!\times \pi_K}\ar@{}[dl]|\swtrans\ar@{}[dl]<-1.25ex>|\id\\
e\ar[r]_-b&B
}}
\end{equation*}
is homotopy exact. Specifically, it is the strict pullback along the Grothendieck opfibration $B\times K\to K$, so is homotopy exact by Proposition~\ref{prop:grothendieckexactsquare}.
\end{proof}

\begin{remark}
Groth in~\cite[Proposition~3.7]{Gro16b} proves out that a derivator is pointed if and only if left Kan extensions commute with right Kan extensions along sieves if and only if right Kan extensions commute with left Kan extensions along cosieves (among other equivalent conditions). We have recovered this result for half derivators as well, or rather, we have recovered the two-thirds of the result that makes sense for each type of half derivator.
\end{remark}

As another ingredient, we would like to know that extension by zero functors commute with cocontinuous morphisms of derivators. In fact, extension by zero functors commute with merely pointed morphisms of derivators, \ie morphisms of prederivators that additionally send zero to zero. This is proven for (two-sided) pointed derivators at~\cite[Corollary~8.2]{Gro16a}, and we give an explicit proof to demonstrate that it still holds for half pointed derivators.

\begin{theorem}\label{thm:cocpreszero}
Suppose that $\Phi\colon\D\to\E$ is a pointed morphism of half pointed derivators. Then $\Phi$ commutes with extension by zero functors (along both sieves and cosieves).
\end{theorem}
\begin{proof}
Let $u\colon J\to K$ be a cosieve, so that $u_!$ is extension by zero. Then we have a canonical transformation
\begin{equation*}
\vcenter{\xymatrix@C=4em{
\D(J)\ar[r]^{u_!}\ar[d]_{\Phi_J}&\D(K)\ar[d]^{\Phi_K}\ar@{}[dl]|\netrans\ar@{}[dl]<-1.8ex>|(.575){\gamma_{u,!}^\Phi}\\
\E(J)\ar[r]_{u_!}&\E(K)
}}
\end{equation*}
which we will prove is an isomorphism. To do so, we paste onto that square: for any $k\in K$, we obtain
\begin{equation*}
\vcenter{\xymatrix@C=4em{
\D(J)\ar[r]^{u_!}\ar[d]_{\Phi_J}&\D(K)\ar[d]^{\Phi_K}\ar@{}[dl]|\netrans\ar@{}[dl]<-1.8ex>|(.575){\gamma_{u,!}^\Phi}\ar[r]^{k^\ast}&\D(e)\ar[d]^{\Phi_e}\ar@{}[dl]<-1.25ex>|(.55)\cong\ar@{}[dl]|\netrans\\
\E(J)\ar[r]_{u_!}&\E(K)\ar[r]_{k^\ast}&\E(e)
}}
\end{equation*}
The added square commutes up to isomorphism because $\Phi$ is a morphism of prederivators. The total pasting depends on whether $k\in J$ or $k\in K\setminus J$. In the first case, $k^\ast u_!\cong k^\ast$, so the total pasting is isomorphic to
\begin{equation*}
\vcenter{\xymatrix@C=4em{
\D(J)\ar[r]^{k^\ast}\ar[d]_{\Phi_J}&\D(e)\ar[d]^{\Phi_e}\ar@{}[dl]<-1.25ex>|(.55)\cong\ar@{}[dl]|\netrans\\
\E(J)\ar[r]_{k^\ast}&\E(e)
}}
\end{equation*}
which is an isomorphism (again because $\Phi$ is a morphism of prederivators). This gives us a composition
\begin{equation*}
\vcenter{\xymatrix@C=4em{
k^\ast u_! \Phi_J\ar@(u,u)[rr]^-{\cong}\ar[r]_-{k^\ast\gamma_{u,!}^\Phi}&k^\ast\Phi_Ku_!\ar[r]_-{\cong}&\Phi_ek^\ast u_!
}}
\end{equation*}
so that $k^\ast\gamma_{u,!}^\Phi$ is an isomorphism for any $k\in J$.

In the second case, $k^\ast u_!= 0$. Specifically, using Der4, we have an isomorphism of functors $k^\ast u_!\cong \pi_{(k/u),!}\pr^\ast$. This yields an isomorphism of pastings
\begin{equation*}
\vcenter{\xymatrix@C=4em{
\D(J)\ar[r]^{u_!}\ar[d]_{\Phi_J}&\D(K)\ar[d]^{\Phi_K}\ar@{}[dl]|\netrans\ar@{}[dl]<-1.8ex>|(.575){\gamma_{u,!}^\Phi}\ar[r]^{k^\ast}&\D(e)\ar[d]^{\Phi_e}\ar@{}[dl]<-1.25ex>|(.55)\cong\ar@{}[dl]|\netrans\\
\E(J)\ar[r]_{u_!}&\E(K)\ar[r]_{k^\ast}&\E(e)
}}
\quad\cong\quad
\vcenter{\xymatrix@C=4em{
\D(J)\ar[r]^{\pr^\ast}\ar[d]_{\Phi_J}&\D(\varnothing)\ar[d]^{\Phi_\varnothing}\ar@{}[dl]|\netrans\ar[r]^{\pi_{\varnothing,!}}&\D(e)\ar[d]^{\Phi_e}\ar@{}[dl]|\netrans\\
\E(J)\ar[r]_{\pr^\ast}&\E(\varnothing)\ar[r]_{\pi_{\varnothing,!}}&\E(e)
}}
\end{equation*}
where we identify the comma category $(k/u)=\varnothing$ when $k\in K\setminus J$ because the functor $u$ is a cosieve. For this righthand diagram, the left square commutes up to isomorphism because $\Phi$ is a morphism of prederivators, and the right square does as well because $\Phi$ is pointed. Thus $k^\ast\gamma_{u_!}^\Phi$ is an isomorphism for $k\in K\setminus J$.

Putting these together, we complete the proof for extensions by zero along cosieves. For extensions by zero along sieves, the proof is the dual of the above. 
\end{proof}

We obtain the following definition, which perfectly generalises the key example of Lemma~\ref{lemma:waldhausenderivator} and sets up the optimal domain for derivator K-theory.
\begin{defn}
Let $\mathbf{Der}_K$ be the 2-category defined as follows: the objects are (strong) left pointed derivators defined on $\Dirf$, the 1-morphisms are cocontinuous morphisms of derivators, and the 2-morphisms are isomodifications.

For any $\D\in\Der_K$ and any $K\in\Dirf$, the prederivator $\D^K$ is in $\Der_K$. For any functor $u\colon J\to K$, the morphisms $u^\ast\colon \D^K\to\D^J$ and~$u_!\colon \D^J\to\D^K$ are in $\Der_K$. For any sieve $i\colon A\to B$, the morphism $i_\ast\colon \D^A\to\D^B$ is in $\Der_K$. 
\end{defn}

An exact functor between Waldhausen categories induces a cocontinuous morphism of derivators by \cite[Lemme~4.3]{Cis10} again. Therefore there is a 2-functor $\mathbf{WaldCat}\to\mathbf{Der}_K$ from the 2-category of saturated Waldhausen categories satisfying the cylinder axiom and exact morphisms (and trivial 2-morphisms). This allows us to compute the derivator K-theory of any such Waldhausen category.


\bibliography{master-bibliography}
\bibliographystyle{alpha}

\end{document}